\documentclass[11pt]{amsart}
\usepackage{hyperref,mathrsfs}
\usepackage{color,graphicx} 
\usepackage{amssymb,latexsym,amsmath}     

\newtheorem{theorem}{Theorem}[section]
\newtheorem{proposition}{Proposition}[section]
\newtheorem{conjecture}{Conjecture}[section]
\newtheorem{lemma}{Lemma}[section]

\begin{document}
\title[Ideal right-angled polyhedra]{Ideal right-angled polyhedra  in Lobachevsky space}
\author{A.Yu. Vesnin, A.A. Egorov}
\address{Novosibirsk State University and Sobolev Institute of Mathematics, Novosibirsk, 630090, Russia; Tomsk State University, 634050, Tomsk, Russia} 
\email{vesnin@math.nsc.ru}
\address{Novosibirsk State University, Novosibirsk, 630090, Russia, Tomsk State University, 634050, Russia} 
\email{a.egorov2@g.nsu.ru }

\begin{abstract} 
In this paper we consider a class of right-angled polyhedra in three-dimensional Lobachevsky space, all vertices of which lie on the absolute. New upper bounds on volumes in terms the number of faces of the polyhedron are obtained. Volumes of polyhedra with at most 23 faces are computed. It is shown that the minimum volumes are realized on antiprisms and twisted antiprisms. The first 248 values of volumes of ideal right-angled polyhedra are presented. Moreover,  the class of polyhedra with isolated triangles is introduces and there are obtained combinatorial bounds on their existence as well as minimal examples of such polyhedra are given.
\end{abstract}
	
\subjclass[2010]{57M25}	
\keywords{Hyperbolic 3-space, ideal polyhedron, right-angled polyhedron, antiprism} 
\thanks{This work was supported in part by the RFBR (project 19-01-00569)}
	
\maketitle 
	
	
\section*{Introduction} \label{section0}

Applying of computer methods is a powerful tool for the study of three-dimensional hyperbolic manifolds.  For example, the tabulation of manifolds obtained by Dehn surgery on manifolds with cusps led by S.\,V.~Matveev and A.\,T.~Fomenko~\cite{MF} and independently J.~Weeks~\cite{Weeks} to recognizing  the smallest volume closed orientable three-dimensional hyperbolic manifold. Todays it is known as  the \emph{Weeks~-- Matveev~-- Fomenko manifold}. Recall that it can be obtained by surgery on the Whitehead link and its volume is approximately equal to $0.942707$. 
	
In recent years many results appeared on the enumeration and classifica\-tion of three-dimensional hyperbolic manifolds which admit decompositions into polyhedra with prescribed properties.
	
In~\cite{FGGTV} there are described hyperbolic three-dimensional manifolds that can be decomposed into regular ideal tetrahedra (up to 25 tetrahedra in the oriented case and up to 22 tetrahedrons in the non-oriented case). Three-dimensional hyperbolic manifolds that can be subdivided into Platonic poly\-hed\-ra are listed in~\cite{Goerner}. In~\cite{191} all three-dimensional orientable manifolds that can be obtained from various realizations of an octahedron were const\-ruc\-ted and classified. The paper~\cite{Inoue} contains an initial list of 825 bounded right-angled hyperbolic polyhedra.
	
In this paper, the objects of our study are polyhedra which can be realized with right, $\pi/2$, dihedral angles in a three-dimensional space of constant negative curvature $\mathbb H^{3}$, known as hyperbolic space or Lobachevsky space. Moreover,  we will consider only ideal right-angled hyperbolic polyhedra, that is, those for which all vertices lie on the absolute of Lobachevsky space. We will  denote by $\mathcal{IR}$ the class of ideal right-angled three-dimensional hyperbolic polyhedra. Recent results on the theory of right-angled polyhedra in Lobachevsky space and using them for constructing three-dimensional hyperbolic manifolds are given in the survey~\cite{Ves}, see also~\cite{Ero}. The main attention in the survey was given to bounded right-angled polyhedra, while in this paper we will consider the case of ideal polyhedra. We will follow the standard terminology of the theory of hyperbolic manifolds; see, for example,~\cite{Ratcliffe}.
	
Hyperbolic three-dimensional manifolds of finite volume, which can be decomposed into ideal right-angled polyhedra, have been intensively studied in last decade. In particular, due to their close relationship with the right-angled Coxeter groups, and the fact that their fundamental groups have the LERF property, i.e. they are locally extended residually finite groups (each finitely generated subgroup is separable)~\cite{Scott}. Several types of hyperbolic three-dimensional manifolds admitting decomposition into ideal right-angled polyhedra are presented in~\cite{CDW}. Since the volume of a manifold is the sum of the volumes of the polyhedra into which it is decomposed, a description of the volumes of ideal right-angled polyhedra is interesting from this point of view.
	
The paper has the following structure. In  Section~\ref{sec1} we recall some facts about the existence of ideal polyhedra in Lobachevsky spaces, in particular, Andreev's theorem (Theorem~\ref{theoremAndreev}) and Rivin's theorem (Theorem~\ref{theoremRivin}), which give necessary and sufficient conditions for their existence in dimension three. In Section~\ref{sec2} the notion of twisted antiprism is introduced and a formula for volumes of right-angled twisted antiprisms is given (Theorem~\ref{theorem2.5}). The Table~\ref{table-2} provides information on the number of ideal right-angled polyhedra in Lobachevsky space with at most $23$ faces and indicates the minimum and maximum volume values for each number of faces. The carried out calculations allow to propose a conjecture which polyhedra are of smallest volumes for an arbitrary number of faces (Conjecture~\ref{hyp2.1}). In Section~\ref{sec3} we obtain new upper bounds on the volume of an ideal right-angled polyhedron in terms of the number of its faces (Theorems~\ref{theorem3.2} and~\ref{theorem3.3}). Also, we present ideal right-angled polyhedra of smallest and largest volume with at most $23$ faces (see Tables~\ref{table-40} and~\ref{table-50}) and the first $248$ values of volumes of ideal right-angled polyhedra (see Table~\ref{table-100}). In Section~\ref{sec4} we introduce the notion of polyhedra with isolated triangles and give a lower bound on the number of faces of such polyhedra (Proposition~\ref{prop4.1}). For the minimum possible number of faces equals to $26$, two examples of polyhedra with isolated triangles are given.
	
\section{Existence} \label{sec1} 
	
\subsection{Dimension and number of cusps}
	
It was shown in~\cite{De} that right-angled polyhedra of finite volume can exist in $\mathbb H^{n} $ only for dimensions $n <13$. Such polyhedra can have both finite vertices and cusps.  At the same time~\cite{Vin84}, in dimensions $n>4$ there are no compact right-angled polyhedra. Thus, a right-angled polyhedron in  $\mathbb H^{n}$, $n\geq 5$, has at least one cusp. It was shown in~\cite{Nonaka}  that in high dimensions right-angled polyhedra of finite volume should have a lot of cusps. Namely, the lower bounds $c(n)$ of the number of cusps for dimensions $n <13$ are given in Table~\ref{table-1}. 
\begin{table}[ht] \caption{Lower bounds for the number of cusps.} \label{table-1} 
\begin{center} 
\begin{tabular}{|c|r|r|r|r|r|r|r|} \hline 
n & 6 & 7& 8& 9& 10 & 11 & 12 \\ \hline
c(n) & 3 & 17 & 36 & 91 & 254 & 741 & 2200 \\ \hline
\end{tabular} 
\end{center}
\end{table}
	
We will be interested in right-angled hyperbolic polyhedra in which all vertices are cusps. Such polyhedra are called \emph{ideal}. 
It was shown in~\cite{Ko2012} that in $\mathbb H^{n}$, $n \geq 7$, there are no ideal right-angled  polyhedra. 
Examples of three-dimensional and four-dimensional ideal right-angled hyperbolic polyhedra will be given below. 
	
\subsection{Three-dimensional case}
	
Necessary and sufficient conditions for a combinatorial polyhedron $P$ to belong to the class $\mathcal IR$ can be obtained as a very special case of E.\,M.~Andreev's theorem~\cite{And70-2} on acute-angled polyhedra of finite volume.
	
\begin{theorem} \cite{And70-2} \label{theoremAndreev}
Let $P$ be an abstract three-dimensional polyhedron with three or four faces meeting at each vertex, and $P$ is not a simplex. The following conditions are necessary and sufficient conditions for the existence in $\mathbb H^{3}$  of a convex polyhedron of finite volume of a combinatorial type $P$ with angles $\alpha_{ij} \leq \pi /2$: 
\begin{itemize}
\item[0.] $0 < \alpha_{ij} \leq \pi/2$. 
\item[1.] If $F_{ijk}$ is a vertex of $P$, then $\alpha_{ij} + \alpha_{jk} + \alpha_{ki} \geq \pi$, if $F_{ijkl}$ is a vertex, then $\alpha_{ij} + \alpha_{jk} + \alpha_{kl} + \alpha_{li} = 2 \pi$.  
\item[2.] If $F_{i}$, $F_{j}$, $F_{k}$ is a triangular prismatic element, then $\alpha_{ij} + \alpha_{jk} + \alpha_{ki} < \pi$.
\item[3.] If $F_{i}$, $F_{j}$, $F_{k}$, $F_{l}$ is a quadrilateral prismatic element, then $\alpha_{ij} + \alpha_{jk} + \alpha_{kl} + \alpha_{li} < 2 \pi$.
\item[4.] If $P$ is a triangular prism with bases $F_{1}$ and $F_{2}$, then $\alpha_{13} + \alpha_{14} + \alpha_{15} + \alpha_{23} + \alpha_{24} + \alpha_{25} < 3 \pi$. 
\item[5.] If among the faces $F_{i}$, $F_{j}$, $F_{k}$ there are adjacent $F_{i}$ and $F_{j}$, $F_{j}$ and $F_{k}$, but $F_{i}$ and $F_{k}$ are not adjacent, but meet at a common vertex and all three faces don't meet at one vertex, then $\alpha_{ij} + \alpha_{jk} < \pi$. 
\end{itemize}
\end{theorem}
	
For the case of an ideal right-angled polyhedron the conditions are significantly simplified since all vertices are four-valent and all dihedral angles are $\pi/2$.
	
We also recall a result of I.~Rivin~\cite{Riv}, concerning arbitrary convex ideal hyperbolic polyhedra, which is formulated in terms of the dual graph.
	
\begin{theorem} \cite{Riv} \label{theoremRivin}
Let $P$ be a plane polyhedral graph with a weight $w(e)$ assigned to each edge $e$. Let $P^{*}$ be the dual graph of $P$ and assume that the weight $w^{*}(e^{*}) = \pi - w(e)$ is assigned to the edge $e^{*}$ dual to $e$. Then $P$ can be realized as a convex ideal polyhedron in $\mathbb H^{3}$  with dihedral angles $w(e)$ at its edges if and only if the following conditions are satisfied:
\begin{itemize}
\item[\rm (1)] $0 < w^{*}(e^{*}) < \pi$ for all $e$; 
\item[\rm (2)] if edges $e^{*}_{1}, e^{*}_{2}, \ldots, e^{*}_{k}$ bound a face in $P^{*}$, then
$$
w^{*}(e^{*}_{1}) + w^{*}(e^{*}_{2}) + \ldots + w^{*}(e^{*}_{k}) = 2 \pi ;
$$
\item[\rm (3)]  if edges $e^{*}_{1}, e^{*}_{2}, \ldots, e^{*}_{k}$ form a cycle in $P^{*}$, that does not bound a face, then
$$
w^{*}(e^{*}_{1}) + w^{*}(e^{*}_{2}) + \ldots + w^{*}(e^{*}_{k}) > 2 \pi .
$$
\end{itemize}
\end{theorem}
	
Ideal hyperbolic polyhedra are also interesting from the point of views of Euclidean geometry, since they are exactly those polyhedra that can be inscribed into the ball. This correspondence and the Rivin's theorem made it possible to solve the problem of Jacob Steiner on describing a sphere around a polyhedron.
	
In the case of a right-angled polyhedron, the weights of edges and the weights of dual edges in the Rivin's theorem are $\pi/2$ and all faces of $P^{*}$ are quadrilateral. Accordingly, all vertices of the polyhedron $P \in \mathcal{IR}$ are 4-valent.
 	
It is easy to see that from the classical Euler formula for a polyhedron, $V-E + F = 2$, where $V$ is the number of vertices, $E$ is the number of edges, and $F$ is the number of faces of a polyhedron $P$, and from the fact that each vertex of an ideal right-angled polyhedron is incident to exactly four edges, it follows that $2E = 4V$.  Thus,  the relation $F = V + 2$ holds. Denote by $p_{k}$ the number of $k$--gonal faces ($k \geq 3$) of a polyhedron of class $\mathcal IR$. Then
\begin{equation} \label{formulamain} 
p_{3} = 8 + \sum_{k\geq 5} p_{k} (k-4).
\end{equation} 
Thus, each polyhedron from $\mathcal {IR}$ has at least $8$ triangular faces. It is easy to see that the octahedron shown in the Figure~\ref{fig1} has the minimum number of faces among the polyhedra in $\mathcal {IR}$.
\begin{figure}[ht]
\centering
\includegraphics[width=0.3\textwidth]{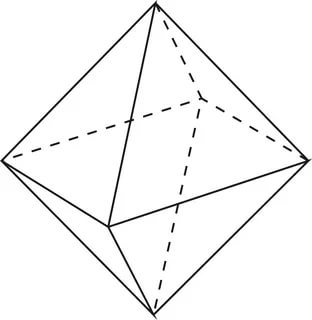}
\caption{The octahedron.} \label{fig1} 
\end{figure}
At the same time, as we will see below, the octahedron is also minimal in volume among all ideal right-angled polyhedra. 
	
\subsection{Four-dimensional case}
	
Recall~\cite[Table I]{Coxeter} that the only regular four-dimensional polyhedron for which each vertex has a type of the cone over a cube is a 24-cell. In particular, it is realized as an ideal right-angled polyhedron in $\mathbb H^{4}$.
	
Let $P \subset \mathbb H^{4}$ be an ideal right-angled polyhedron with the face vector  $f(P) = (f_{0}, f_{1}, f_{2}, f_{3})$. Since $P$ is a convex four-dimensional polyhedron, then its surface $\partial P$ is homeomorphic to $S^{3}$, that means that its Euler characteristic turns to zero. Thus, $f_{0} - f_{1} + f_{2} - f_{3} = 0$. The volume of the polyhedron can expressed in terms of the components of the face vector, namely, the following statement holds. 
	
\begin{lemma} \cite{Ko2012}
Let $P \subset \mathbb H^{3}$ be an ideal right-angled polyhedron with a face vector $f(P) = (f_{0}, f_{1}, f_{2}, f_{3})$. Then its volume is equal to
$$
\operatorname{vol} P = \frac{f_{0} - f_{3} +4}{3} \pi^{2}. 
$$
\end{lemma}
	
Since for a 24-cell, the face vector has the form $(24,96,96,24)$, its volume is $4 \pi^{2} /3$. It is shown in~\cite{Ko2012} that the 24-cell is the unique smallest volume ideal right-angled polyhedron in $\mathbb H^{4}$.
	
\section{Enumeration of polyhedra and their volumes} \label{sec2} 
	
\subsection{Antiprisms and the edge-twist operation} 

Considering the octa\-hed\-ron as a triangular antiprism, that is, a polyhedron with triangular top and bottom and with a lateral surface formed by two levels of triangles, it can be naturally generalized to the next infinite family of polyhedra. By an $n$-antiprism $A(n)$, $n \geq 3$, we mean $(2n+2)$-hedron with $n$-gonal top and bottom and with a lateral surface formed by two levels of $n$ triangles in each; and with each vertex incident to four edges. Schlegel diagrams of polyhedra $A(3)$ and $A(4)$ are presented in Figure~\ref{fig200}. 

By checking the conditions of Andreev's theorem or Rivin's theorem, it is easy to see that the antiprisms $A (n)$, $n \geq 3$, can be realized as ideal right-angled polyhedra in $\mathbb H^{3}$.
	
A.~Kolpakov demonstrated in~\cite{Ko2012} that the polyhedra $A(n)$ are minimal in the following sense
	
\begin{theorem} \cite{Ko2012}  \label{theorem2.1}
For $n \geq 3$ the antiprism $A(n)$ has the smallest number of faces, equal to $2n + 2$, among all the ideal right-angled hyperbolic polyhedra with at least one $n$-gonal face.
\end{theorem}
	
As will be clear below, antiprisms play an important role in understanding the structure of the set $\mathcal{IR}$ of all ideal right-angled hyperbolic polyhedra.
	
Theorem~\ref{theoremRivin} admits to characterize the class $\mathcal{IR}$ in terms of the dual graphs of the polyhedral graphs. The dual graph must be a quadrangulation of the sphere, that is, a finite graphs with quadrilateral faces on a 2-sphere. Furthermore, the dual graph cannot contain a cycle of length four that separates two faces. The class of graphs with these properties was considered in~\cite{Bri}, where it was denoted by $\mathscr Q_{4}$. Using the results of~\cite{Bri} on graphs from the class $\mathscr Q_{4}$ and passing from dual graphs to the original graphs, one can state the following result.
	
\begin{theorem} \cite{Bri} \label{theorem2.2} 
The class of $4$-valent $3$-connected and cyclically $6$-connec\-ted planar graphs is generated by 1-skeletons of antiprisms $A(n)$ and by edge-twist moves.
\end{theorem}
	
Recall that a graph is said to be \emph{cyclically $k$-connected} if $k$ is the smallest number of edges such that removing them decomposes the graph into two components each of which contains a cycle.
	
The \emph{edge-twist} move is defined as follows. Let $P \in \mathcal{IR}$ and assume that some face of the polyhedron has four distinct ideal vertices that are pairwise connected by edges  $e_{1}$ and $e_{2}$ as in Figure~\ref{figtwist}.  
\begin{figure}[h] 
\begin{center}
\setlength{\unitlength}{0.34mm}
\begin{picture}(0,50)(0,-10)
\thicklines 
\put(-50,0){\begin{picture}(0,40)
\qbezier(-18,0)(-18,0)(18,0)
\qbezier(-18,30)(-18,30)(18,30)
\put(-20,0){\circle*{4}}
\put(20,0){\circle*{4}}
\put(-20,30){\circle*{4}}
\put(20,30){\circle*{4}}
\put(0,35){\makebox(0,0)[cc]{$e_{1}$}}
\put(0,-5){\makebox(0,0)[cc]{$e_{2}$}}
\put(0,-20){\makebox(0,0)[cc]{$P$}}
\end{picture} }
\put(50,0){\begin{picture}(0,40) 
\thicklines
\qbezier(-18,2)(-18,2)(-2,13)
\qbezier(-18,28)(-18,28)(-2,17)
\qbezier(18,2)(18,2)(2,13)
\qbezier(18,28)(18,28)(2,17)
\put(-20,0){\circle*{4}}
\put(20,0){\circle*{4}}
\put(-20,30){\circle*{4}}
\put(20,30){\circle*{4}}
\put(0,15){\circle*{4}}
\put(0,8){\makebox(0,0)[cc]{$v$}}
\put(0,-20){\makebox(0,0)[cc]{$P^{*}$}}
\end{picture}}
\end{picture} 
\end{center}
\caption{An edge-twist move.} \label{figtwist}
\end{figure}
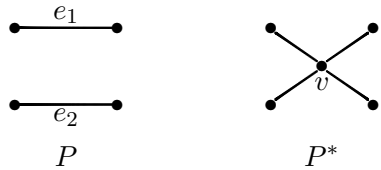

Then the transformation involves removing $e_{1}$ and $e_{2}$, creating a new vertex $v$ and connecting it by edges with the above four vertices. We denote the resulting polyhedron by $P^{*}$ and say that $P^{*}$ is obtained from $P$ by an edge-twist move. For example, $A(4)^{*}$ in Figure~\ref{fig200} is obtained from  $A(4)$  by the edge-twist move. The edges $e_{1}$, $e_{2}$ and the new vertex $v$ are indicated in the figure. 
\begin{figure}[h]
	\begin{center} 
		\unitlength=0.4mm
		\begin{picture}(0,80) 
		\put(-100,0){\begin{picture}(0,80)(0,-5)
			\thicklines 
			\qbezier(-20,20)(-20,20)(20,20)
			\qbezier(-20,20)(-20,20)(0,50) 
			\qbezier(20,20)(20,20)(0,50) 
			\qbezier(0,-5)(0,-5)(-20,20)
			\qbezier(0,-5)(0,-5)(20,20) 
			\qbezier(-30,50)(-30,50)(-20,20)
			\qbezier(-30,50)(-30,50)(0,50) 
			\qbezier(30,50)(30,50)(20,20)
			\qbezier(30,50)(30,50)(0,50)
			\qbezier(-30,50)(-50,0)(0,-5)
			\qbezier(30,50)(50,0)(0,-5)
			\qbezier(-30,50)(0,90)(30,50)
			\end{picture}
		}
		\put(0,0){\begin{picture}(0,80)(0,0)
			\thicklines 
			\qbezier(-20,20)(-20,20)(20,20)
			\qbezier(-20,20)(-20,20)(-20,60) 
			\qbezier(-20,60)(-20,60)(20,60)
			\qbezier(20,20)(20,20)(20,60) 
			\qbezier(0,0)(0,0)(-20,20)
			\qbezier(0,0)(0,0)(20,20) 
			\qbezier(-40,40)(-40,40)(-20,20)
			\qbezier(-40,40)(-40,40)(-20,60) 
			\qbezier(40,40)(40,40)(20,20)
			\qbezier(40,40)(40,40)(20,60)
			\qbezier(0,80)(0,80)(-20,60)
			\qbezier(0,80)(0,80)(20,60)
			\qbezier(-40,40)(-40,0)(0,0)
			\qbezier(40,40)(40,0)(0,0)
			\qbezier(-40,40)(-40,80)(0,80)
			\qbezier(40,40)(40,80)(0,80)
			\put(-15,40){\makebox(0,0)[cc]{$e_{1}$}}
			\put(15,40){\makebox(0,0)[cc]{$e_{2}$}}
			\end{picture}
		}
		\put(100,0){\begin{picture}(0,80)(0,0)
			\thicklines 
			\qbezier(-20,20)(-20,20)(20,20)
			\qbezier(-20,60)(-20,60)(20,60)
			\qbezier(0,40)(0,40)(-20,20) 
			\qbezier(0,40)(0,40)(20,20) 
			\qbezier(0,40)(0,40)(-20,60) 
			\qbezier(0,40)(0,40)(20,60) 
			\qbezier(0,0)(0,0)(-20,20)
			\qbezier(0,0)(0,0)(20,20) 
			\qbezier(-40,40)(-40,40)(-20,20)
			\qbezier(-40,40)(-40,40)(-20,60) 
			\qbezier(40,40)(40,40)(20,20)
			\qbezier(40,40)(40,40)(20,60)
			\qbezier(0,80)(0,80)(-20,60)
			\qbezier(0,80)(0,80)(20,60)
			\qbezier(-40,40)(-40,0)(0,0)
			\qbezier(40,40)(40,0)(0,0)
			\qbezier(-40,40)(-40,80)(0,80)
			\qbezier(40,40)(40,80)(0,80)
			\put(-8,40){\makebox(0,0)[cc]{$v$}}
			\end{picture}
		}
	\end{picture}
\end{center} \caption{Polyhedra $A(3)$, $A(4)$ and $A(4)^{*}$.} \label{fig200}
\end{figure}
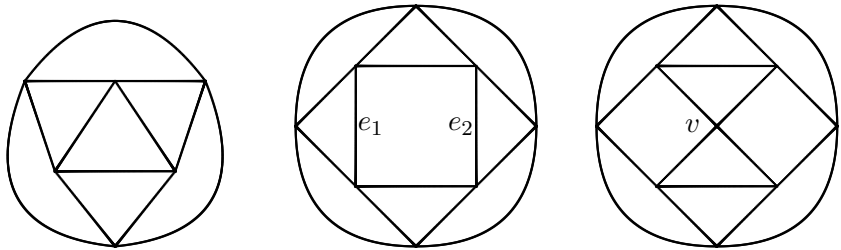

Theorems~\ref{theoremRivin} and~\ref{theorem2.2} lead to the following result.

\begin{theorem} \label{theorem2.3} 
Each ideal right-angled hyperbolic polyhedron is an antiprism or can be obtained from some antiprism by a finite number of edge-twist moves.
\end{theorem}

Let us  introduce a class of polyhedra obtained from antiprisms. Let $A (n)$, $n \geq 4$, be an antiprism, and $e_{1}$ and $e_{2}$ be two edges that belong to  one of $n$-gonal faces, such that there is a third edge on the same face to which they are both adjacent, let us denote it by  $e_{3}$. In other words, $e_{1}$ and $e_{2}$ \emph{are adjacent through an edge}. We apply the edge-twist move to the edges $e_{1}$ and $e_{2}$.  As one can see from Figure~\ref{fig300}, illustrating the case of the antiprism $A (6)$, when we apply edge-twist move  to the edges adjacent through an edge, the combinatorial structure changes as follows. The antiprism $A (n)$ had $2n+2$ faces: two $n$-gonal faces and $2n$ triangular faces. The new polyhedron $A(n)^{*}$ has $2n+3$ faces: one $n$-gonal face, one $(n-1)$-gonal face, two quadrilateral faces and $(2n-1)$ triangular faces. The polyhedron $A (n)^{*}$ will be referred to as a  \emph{twisted antiprism}.
Observe that the number of faces of the antiprism is always even, but  the number of faces of the twisted antiprism is always odd. As well as the antiprism, by virtue of Andreev's or Rivin's theorems, any twisted antiprism can be realized as an ideal right-angled polyhedron in~$\mathbb H^{3}$.
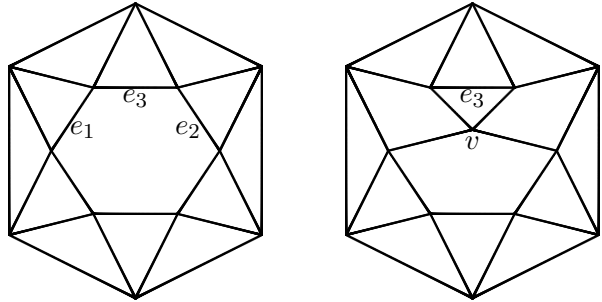
\begin{figure}[h]
\begin{center} 
	\unitlength=0.28mm
	\begin{picture}(0,140) 
	\put(-80,0){\begin{picture}(0,140)(0,0)
		\thicklines 
		\qbezier(0,0)(0,0)(-60,30)
		\qbezier(0,0)(0,0)(-20,40) 
		\qbezier(0,0)(0,0)(60,30)
		\qbezier(0,0)(0,0)(20,40) 
		\qbezier(-60,30)(-60,30)(-60,110)
		\qbezier(-60,30)(-60,30)(-20,40) 
		\qbezier(-60,30)(-60,30)(-40,70)
		\qbezier(60,30)(60,30)(60,110)
		\qbezier(60,30)(60,30)(20,40) 
		\qbezier(60,30)(60,30)(40,70)
		\qbezier(-60,110)(-60,110)(0,140) 
		\qbezier(-60,110)(-60,110)(-20,100)
		\qbezier(-60,110)(-60,110)(-40,70) 
		\qbezier(60,110)(60,110)(0,140) 
		\qbezier(60,110)(60,110)(20,100)
		\qbezier(60,110)(60,110)(40,70) 
		\qbezier(-20,40)(-20,40)(20,40) 
		\qbezier(-20,100)(-20,100)(20,100)
		\qbezier(-20,40)(-20,40)(-40,70)
		\qbezier(-60,110)(-60,110)(-40,70) 
		\qbezier(-40,70)(-40,70)(-20,100) 
		\qbezier(-20,100)(-20,100)(0,140) 
		\qbezier(20,40)(20,40)(40,70)
		\qbezier(60,110)(60,110)(40,70) 
		\qbezier(40,70)(40,70)(20,100)
		\qbezier(20,100)(20,100)(0,140)
		\put(-25,80){\makebox(0,0)[cc]{$e_{1}$}}
		\put(25,80){\makebox(0,0)[cc]{$e_{2}$}}
		\put(0,95){\makebox(0,0)[cc]{$e_{3}$}}
		\end{picture}
	}
	\put(80,0){\begin{picture}(0,140)(0,0)
		\thicklines 
		\qbezier(0,0)(0,0)(-60,30)
		\qbezier(0,0)(0,0)(-20,40) 
		\qbezier(0,0)(0,0)(60,30)
		\qbezier(0,0)(0,0)(20,40) 
		\qbezier(-60,30)(-60,30)(-60,110)
		\qbezier(-60,30)(-60,30)(-20,40) 
		\qbezier(-60,30)(-60,30)(-40,70)
		\qbezier(60,30)(60,30)(60,110)
		\qbezier(60,30)(60,30)(20,40) 
		\qbezier(60,30)(60,30)(40,70)
		\qbezier(-60,110)(-60,110)(0,140) 
		\qbezier(-60,110)(-60,110)(-20,100)
		\qbezier(-60,110)(-60,110)(-40,70) 
		\qbezier(60,110)(60,110)(0,140) 
		\qbezier(60,110)(60,110)(20,100)
		\qbezier(60,110)(60,110)(40,70) 
		\qbezier(-20,40)(-20,40)(20,40) 
		\qbezier(-20,100)(-20,100)(20,100)
		\qbezier(-20,40)(-20,40)(-40,70)
		\qbezier(-60,110)(-60,110)(-40,70) 
		\qbezier(-20,100)(-20,100)(0,140) 
		\qbezier(20,40)(20,40)(40,70)
		\qbezier(60,110)(60,110)(40,70) 
		\qbezier(20,100)(20,100)(0,140)
		\qbezier(0,80)(0,80)(-40,70)
		\qbezier(0,80)(0,80)(-20,100)
		\qbezier(0,80)(0,80)(40,70)
		\qbezier(0,80)(0,80)(20,100)
		\put(0,74){\makebox(0,0)[cc]{$v$}}
		\put(0,95){\makebox(0,0)[cc]{$e_{3}$}}
		\end{picture}
	}
\end{picture}
\end{center} \caption{Polyhedra $A(6)$ and $A(6)^{*}$.} \label{fig300}
\end{figure}

Denote by $v_ {8}$ the volume of an ideal right-angled hyperbolic octahedron. The numerical value of this quantity will be given in the next section.

\begin{lemma} \label{lemma2.1} 
Let $A(n)$, $n \geq 4$, be an ideal right-angled antiprism in $\mathbb H^{3}$ having $2n+2$ faces and $A(n)^{*}$ be an ideal right-angled twisted antiprism obtained from $A(n)$ and having $2n + 3$ faces. Then for the volume of the twisted antiprism the following equality holds:
$$
\operatorname{vol} (A(n)^{*}) = \operatorname{vol} (A(n-1)) + v_{8}. 
$$ 
\end{lemma}

\begin{proof} 
Since the edge-twist move of edges adjacent through one edge is a local transformation of a polyhedron, we will illustrate the proof for the case $n=6$. For an arbitrary $n \geq 4$, the proof is analogous.  

Let us consider the antiprism $A(n-1)$. The left side of Figure~\ref{fig400} presents the antiprism $A(5)$. Put an ideal right-angled octahedron on one of its faces to the triangular face $ABC$. Since the triangular faces of the antiprism and the triangular faces of the octahedron are ideal triangles, these faces are pairwise isometric. The remaining seven faces of the octahedron are drawn inside the triangle $ABC$ in the middle of Figure~\ref {fig400}. Recall that the dihedral angles at the edges of the antiprism and at the edges of the octahedron are equal to $\pi/2$. Therefore, when we combine the antiprism and the octahedron along the triangular face $ABC$ in the resulting polyhedron, the angles at the edges $AB$, $BC$ and $AC$ will be equal to $\pi $ and the corresponding faces will belong to the same plane in pairs. Thus the polyhedron obtained by combining the antiprism and the octahedron will have a combinatorial structure as in the right part of Figure~\ref{fig400}, where the edges $AB$, $BC$ and $AC$ are absent.
\begin{figure}[h]
\begin{center} 
	\unitlength=0.28mm
	\begin{picture}(0,120) 
	\put(-160,0){\begin{picture}(0,110)(0,0)
		\thicklines 
		\qbezier(0,0)(0,0)(-60,30)
		\qbezier(0,0)(0,0)(-20,40) 
		\qbezier(0,0)(0,0)(60,30)
		\qbezier(0,0)(0,0)(20,40) 
		\qbezier(-60,30)(-60,30)(-60,110)
		\qbezier(-60,30)(-60,30)(-20,40) 
		\qbezier(-60,30)(-60,30)(-40,70)
		\qbezier(60,30)(60,30)(60,110)
		\qbezier(60,30)(60,30)(20,40) 
		\qbezier(60,30)(60,30)(40,70)
		\qbezier(-60,110)(-60,110)(-40,70) 
		\qbezier(60,110)(60,110)(40,70) 
		\qbezier(-20,40)(-20,40)(20,40) 
		\qbezier(-20,40)(-20,40)(-40,70)
		\qbezier(-60,110)(-60,110)(-40,70)  
		\qbezier(20,40)(20,40)(40,70)
		\qbezier(60,110)(60,110)(40,70) 
		\qbezier(0,80)(0,80)(-40,70)
		\qbezier(0,80)(0,80)(40,70)
		\qbezier(0,80)(0,80)(-60,110)
		\qbezier(0,80)(0,80)(60,110)
		\qbezier(-60,110)(-60,110)(60,110)
		\put(-66,110){\makebox(0,0)[cc]{$A$}}
		\put(66,110){\makebox(0,0)[cc]{$B$}}
		\put(0,74){\makebox(0,0)[cc]{$C$}}
		\end{picture}
	}
	\put(0,0){\begin{picture}(0,110)(0,0)
		\thicklines 
		\qbezier(0,0)(0,0)(-60,30)
		\qbezier(0,0)(0,0)(-20,40) 
		\qbezier(0,0)(0,0)(60,30)
		\qbezier(0,0)(0,0)(20,40) 
		\qbezier(-60,30)(-60,30)(-60,110)
		\qbezier(-60,30)(-60,30)(-20,40) 
		\qbezier(-60,30)(-60,30)(-40,70)
		\qbezier(60,30)(60,30)(60,110)
		\qbezier(60,30)(60,30)(20,40) 
		\qbezier(60,30)(60,30)(40,70)
		\qbezier(-60,110)(-60,110)(-40,70) 
		\qbezier(60,110)(60,110)(40,70) 
		\qbezier(-20,40)(-20,40)(20,40) 
		\qbezier(-20,40)(-20,40)(-40,70)
		\qbezier(-60,110)(-60,110)(-40,70)  
		\qbezier(20,40)(20,40)(40,70)
		\qbezier(60,110)(60,110)(40,70) 
		\qbezier(0,80)(0,80)(-40,70)
		\qbezier(0,80)(0,80)(40,70)
		\qbezier(0,80)(0,80)(-60,110)
		\qbezier(0,80)(0,80)(60,110)
		\qbezier(-60,110)(-60,110)(60,110)
		\put(-66,110){\makebox(0,0)[cc]{$A$}}
		\put(66,110){\makebox(0,0)[cc]{$B$}}
		\put(0,74){\makebox(0,0)[cc]{$C$}}
		\qbezier(0,105)(0,105)(-60,110)
		\qbezier(0,105)(0,105)(60,110)
		\qbezier(0,105)(0,105)(-10,95)
		\qbezier(0,105)(0,105)(10,95)
		\qbezier(-60,110)(-60,110)(-10,95)
		\qbezier(60,110)(60,110)(10,95)
		\qbezier(-10,95)(-10,95)(10,95)
		\qbezier(0,80)(0,80)(-10,95)
		\qbezier(0,80)(0,80)(10,95)
		\end{picture}
	}
	\put(160,0){\begin{picture}(0,110)(0,0)
		\thicklines 
		\qbezier(0,0)(0,0)(-60,30)
		\qbezier(0,0)(0,0)(-20,40) 
		\qbezier(0,0)(0,0)(60,30)
		\qbezier(0,0)(0,0)(20,40) 
		\qbezier(-60,30)(-60,30)(-60,110)
		\qbezier(-60,30)(-60,30)(-20,40) 
		\qbezier(-60,30)(-60,30)(-40,70)
		\qbezier(60,30)(60,30)(60,110)
		\qbezier(60,30)(60,30)(20,40) 
		\qbezier(60,30)(60,30)(40,70)
		\qbezier(-60,110)(-60,110)(-40,70) 
		\qbezier(60,110)(60,110)(40,70) 
		\qbezier(-20,40)(-20,40)(20,40) 
		\qbezier(-20,40)(-20,40)(-40,70)
		\qbezier(-60,110)(-60,110)(-40,70)  
		\qbezier(20,40)(20,40)(40,70)
		\qbezier(60,110)(60,110)(40,70) 
		\qbezier(0,80)(0,80)(-40,70)
		\qbezier(0,80)(0,80)(40,70)
		\qbezier(0,105)(0,105)(-60,110)
		\qbezier(0,105)(0,105)(60,110)
		\qbezier(0,105)(0,105)(-10,95)
		\qbezier(0,105)(0,105)(10,95)
		\qbezier(-60,110)(-60,110)(-10,95)
		\qbezier(60,110)(60,110)(10,95)
		\qbezier(-10,95)(-10,95)(10,95)
		\qbezier(0,80)(0,80)(-10,95)
		\qbezier(0,80)(0,80)(10,95)
		\end{picture}
	}
\end{picture}
\end{center} \caption{ $A(5)$  and its union with an octahedron.} \label{fig400}
\end{figure}
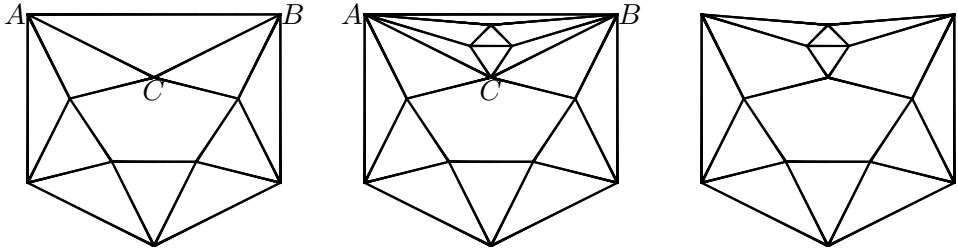
It is easy to see that the resulting polyhedron coincides combinatorially with the polyhedron $A(6)^{*}$ and is also right-angled.
\end{proof} 

\subsection{The Lobachevsky function} 

Traditionally, volumes of polyhedra in three-dimensional hyperbolic space are computed in terms of function
$$
\Lambda(\theta) = - \int\limits_0^{\theta} \log | 2 \sin (t) | \, {\rm d} t , \label{lob}
$$
which J.~Milnor introduced in the survey~\cite{Mil} and called it the \emph{Lobachevsky function}. He demonstrated that the volume of an ideal tetrahedron $T(\alpha, \beta, \gamma)$ in a three-dimensional hyperbolic space with dihedral angles $\alpha$, $\beta$ and $\gamma$ (meaning that at one of the vertices dihedral angles are $\alpha$, $\beta$, $\gamma$, where $\alpha + \beta + \gamma = \pi$, and dihedral angles at the opposite edges of the tetrahedron coincide) is calculated by the formula:
$$
\operatorname{vol} (T(\alpha, \beta, \gamma)) = \Lambda (\alpha) + \Lambda (\beta) + \Lambda (\gamma). 
$$
Splitting an ideal right-angled octahedron into four ideal tetrahedra with dihedral angles $\alpha = \pi/2$, $\beta = \pi/4$, $\gamma=\pi/4$, and using $\Lambda (\pi/2) =0$, we obtain
$$
v_{8} = 8 \Lambda \left( \frac{\pi}{4} \right) = 3.663862376708876\dots . . 
$$

A formula for the volume of a right-angled antiprism was presented by W.~Thurston in his well-known lectures~\cite[Chapter~6.8]{Th78}, where the antiprism was called by a drum with triangular sides, and its volume was used to calculate the volume of the complement to some link in a three-dimensional sphere.

\begin{theorem} \cite{Th78} \label{theorem2.4}
For $n \geq 3$ the volume of a right-angled $n$--antiprism is given by 
$$
\operatorname{vol} (A(n)) = 2 n \, \left[ \Lambda \left( \frac{\pi}{4} + \frac{\pi}{2n} \right) + \Lambda \left( \frac{\pi}{4} - \frac{\pi}{2n} \right) \right] ,
$$
where $\Lambda (x)$ is the Lobachevsky function. 
\end{theorem}

Theorem~\ref{theorem2.4} implies that $\operatorname{vol} (A(n))$ is asymptotically equivalent to $\frac{v_{8}}{2} n$ when $n \to \infty$. In particular, the above formula gives the volume of an ideal right-angled octahedron $A (3)$:
$$
\operatorname{vol} (A(3)) = 8 \Lambda ( \pi /4 ) = v_{8}, 
$$
where we used properties of the Lobachevsky function~\cite{Ratcliffe}. 

\begin{theorem} \label{theorem2.5} 
For the volume of the twisted antiprism  $A(n)^{*}$, $n \geq 4$, the following formula holds: 
$$
\operatorname{vol} (A(n)^{*})     
= 2(n-1) \, \left[ \Lambda \left( \frac{\pi}{4} + \frac{\pi}{2(n-1)} \right) + \Lambda \left( \frac{\pi}{4} - \frac{\pi}{2(n-1)} \right) \right] + 8 \Lambda \left( \frac{\pi}{4} \right).
$$
\end{theorem}

\begin{proof}
It follows from Lemma~\ref{lemma2.1} and Theorem~\ref{theorem2.4}.
\end{proof}

\subsection{Volumes of polyhedra with at most 23 faces}

As we noted above, the Euler's formula implies that any ideal right-angled polyhedron has at least $8$ triangular faces.

For each $n = 8, 10, 11, \ldots, 23$ Table~\ref{table-2} gives the number of ideal right-angled polyhedra in $\mathbb H^{3}$ with $n$ faces; the number of different volumes of them; the minimum and maximum volume values. 
\begin{table}[ht] \caption{Ideal right-angled polyhedra.} \label{table-2} 
\begin{center} 
\begin{tabular}{|r|r|r|r|r|} \hline
$\#$  of faces & $\#$ of polyhedra & $\#$ of volumes & $\min$ volume & $\max$ volume \\ \hline \hline 
8 & 1 & 1 & 3.663863  & 3.663863\\ \hline 
9 & 0 & 0 & - & -\\ \hline
10 & 1 & 1 & 6.023046 & 6.023046 \\ \hline
11 & 1 & 1 & 7.327725 & 7.327725 \\ \hline
12 & 2 & 2 & 8.137885  & 8.612415  \\ \hline
13 & 2 & 2 & 9.686908 & 10.149416   \\ \hline
14 & 9 & 7 &  10.149416 & 12.046092  \\ \hline
15 & 11 & 7 &  11.801747 & 13.350771   \\ \hline
16 & 37 & 17 & 12.106298 & 14.832681   \\ \hline
17 & 79 & 31 & 13.813278 & 16.331571   \\ \hline
18 & 249 & 79 & 14.030461 & 18.069138   \\ \hline
19 & 671 & 172 & 15.770160  & 19.523353  \\ \hline
20 & 2182 & 495 & 15.933385  & 21.241543  \\ \hline
21 & 6692 & 1359 & 17.694323 & 22.894415   \\ \hline
22 & 22131 & 4276 & 17.821704 & 24.599233   \\ \hline
23 & 72405 & 13031 & 19.597248 & 26.228126  \\ \hline
\end{tabular}
\end{center}
\end{table}

Combinatorial enumeration of polyhedra was done with using the computer program {\tt plantri}~\cite{plantri}. Volumes of hyperbolic polyhedra were calculated with the applying some modification of the computer program  {\tt Snap\-Pea}~\cite{SnapPea}.  Values of volumes are given up to $10^{-6}$. 

Denote by  $P_{n}^{\min}$ and $P_{n}^{\max}$  polyhedra which realize minimal and maximal values of volumes in the class of all polyhedra with $n$ faces.  The first seven polyhedra with smallest volume are given in Table~\ref{table-3}.
\begin{table}[!ht]
\caption{The first seven ideal right-angled polyhedra} \label{table-3}
\begin{tabular}{ccc} 
\includegraphics[width=0.3\textwidth]{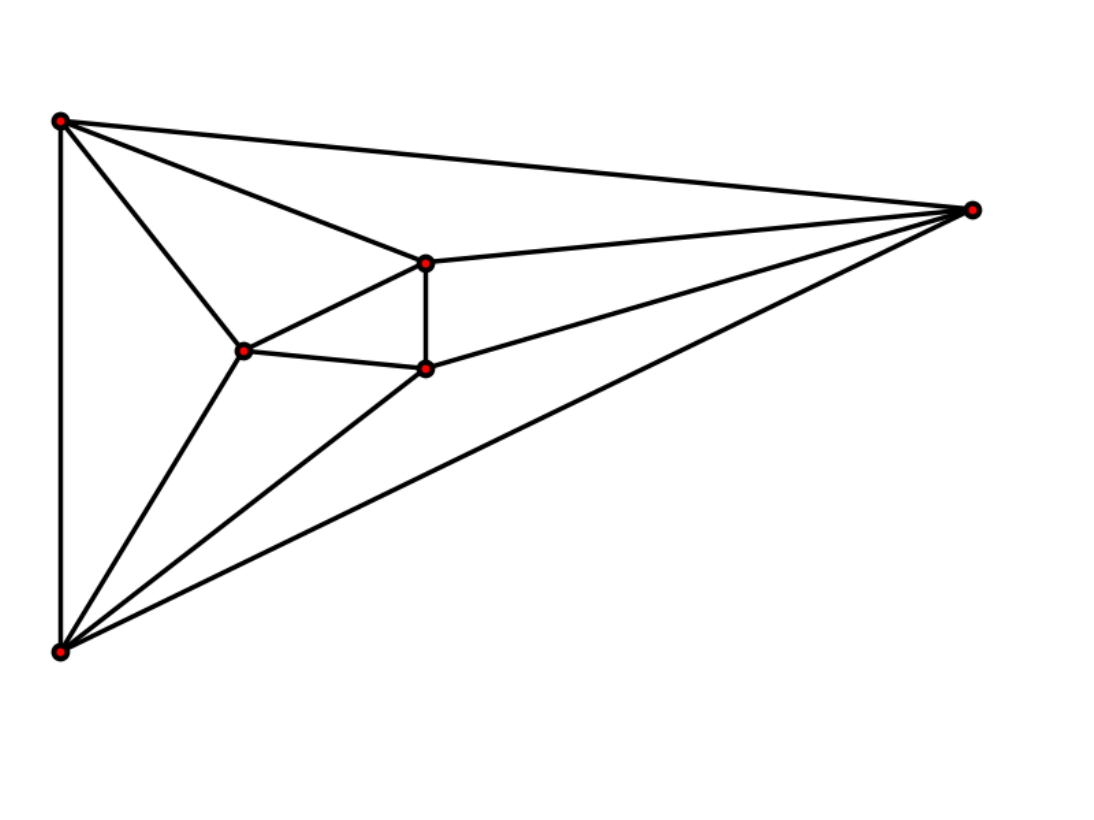} & 
\includegraphics[width=0.3\textwidth]{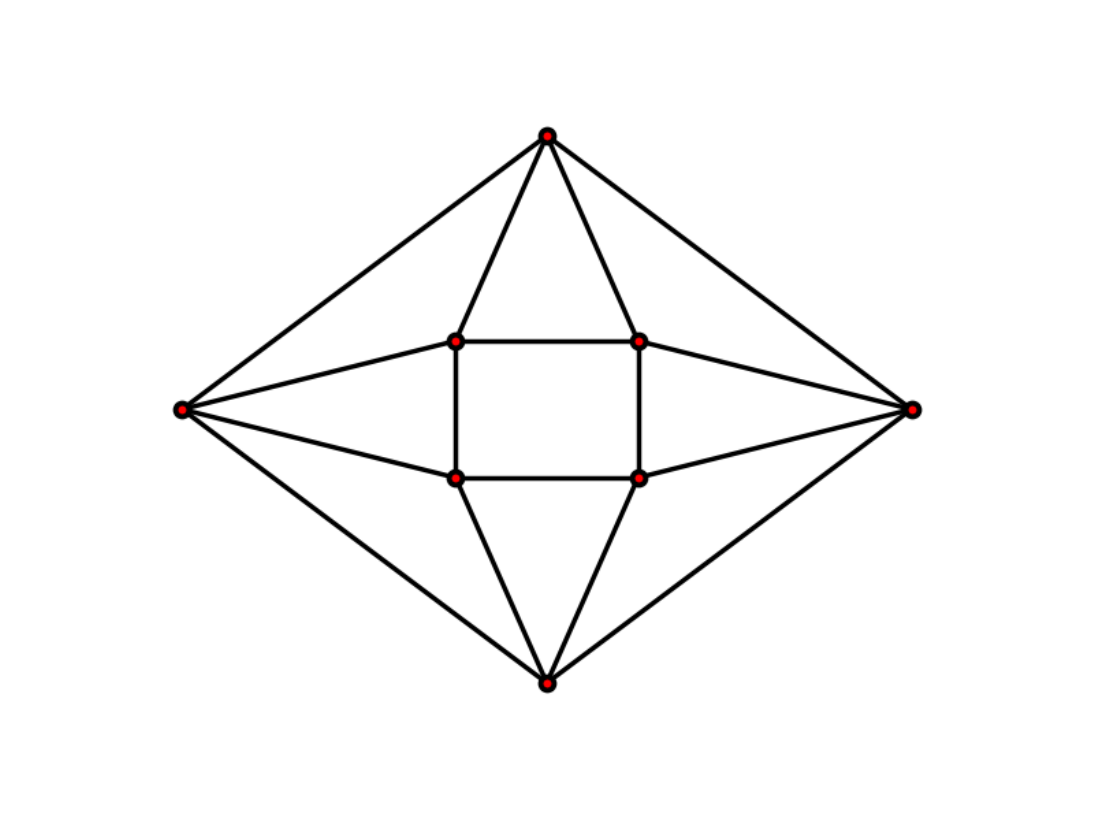} &
\includegraphics[width=0.3\textwidth]{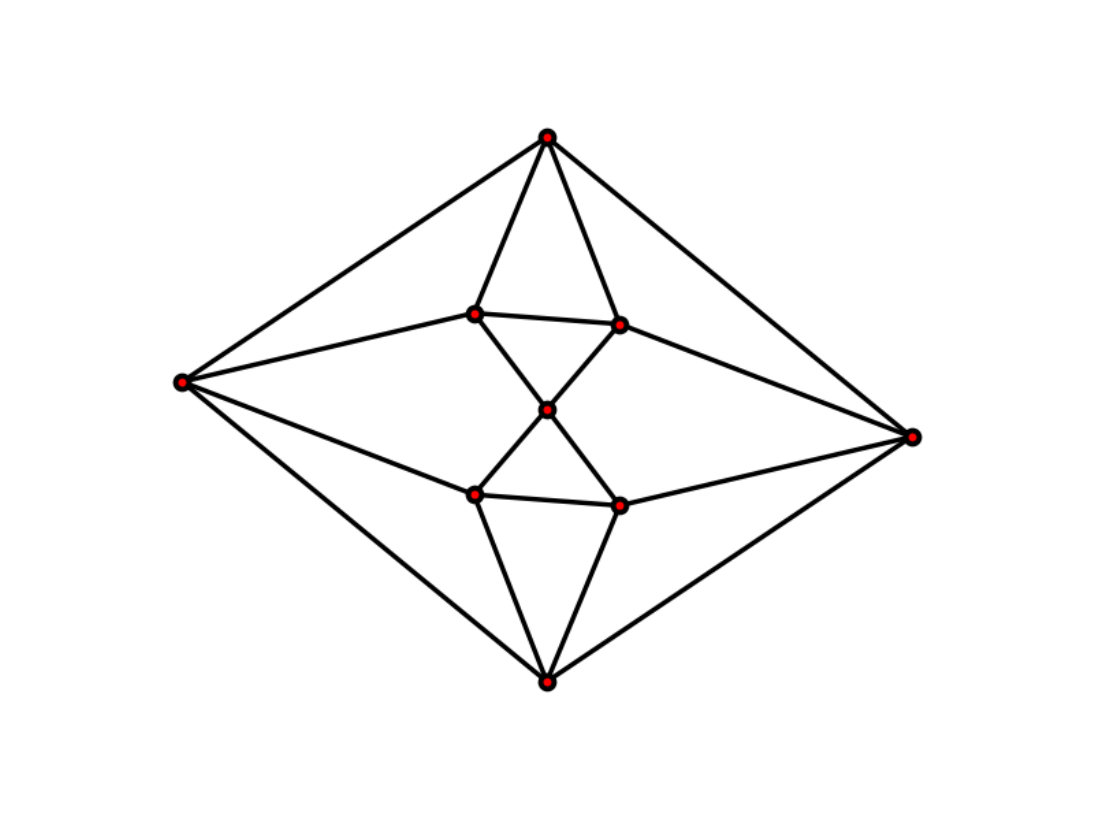} \\ 
(1) $\operatorname{vol} (P_{8}) = 3.663863$ & (2) $\operatorname{vol} (P_{10}) = 6.023046$ &   (3) $\operatorname{vol} (P_{11}) = 7.327725$ \\ 
\end{tabular}
\begin{tabular}{cc} 
\includegraphics[width=0.4\textwidth]{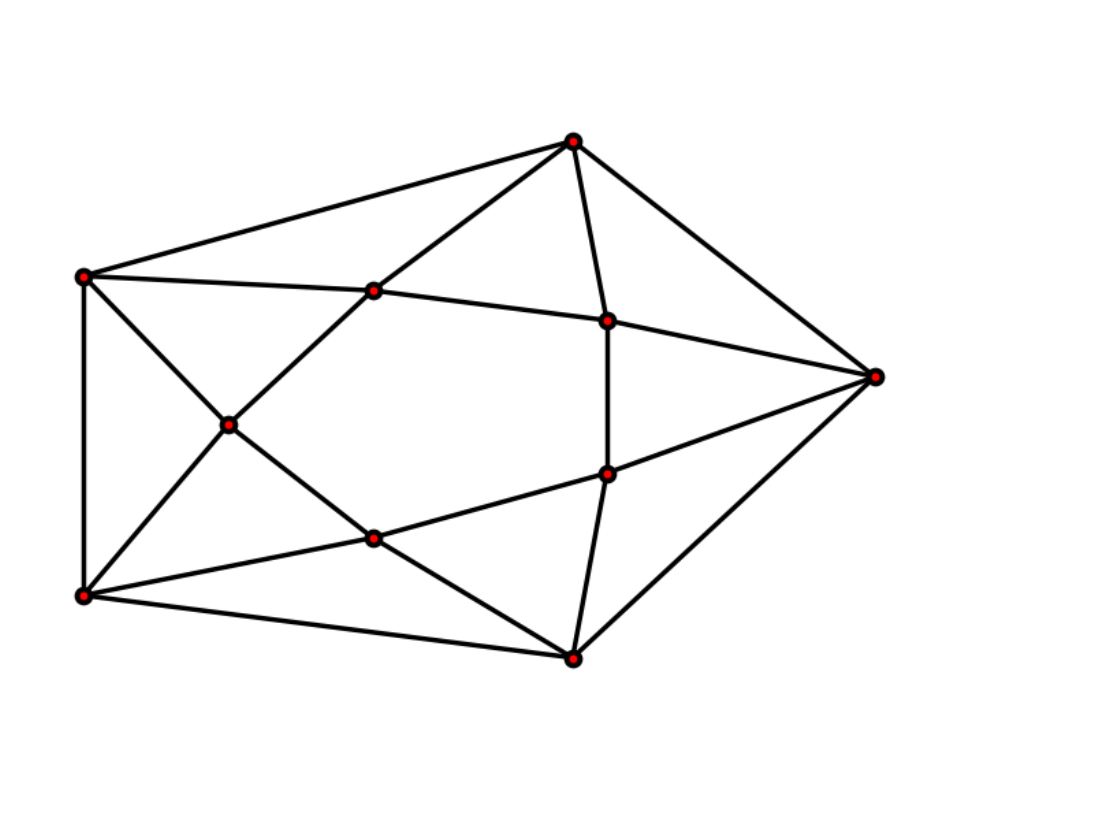} &  
\includegraphics[width=0.4\textwidth]{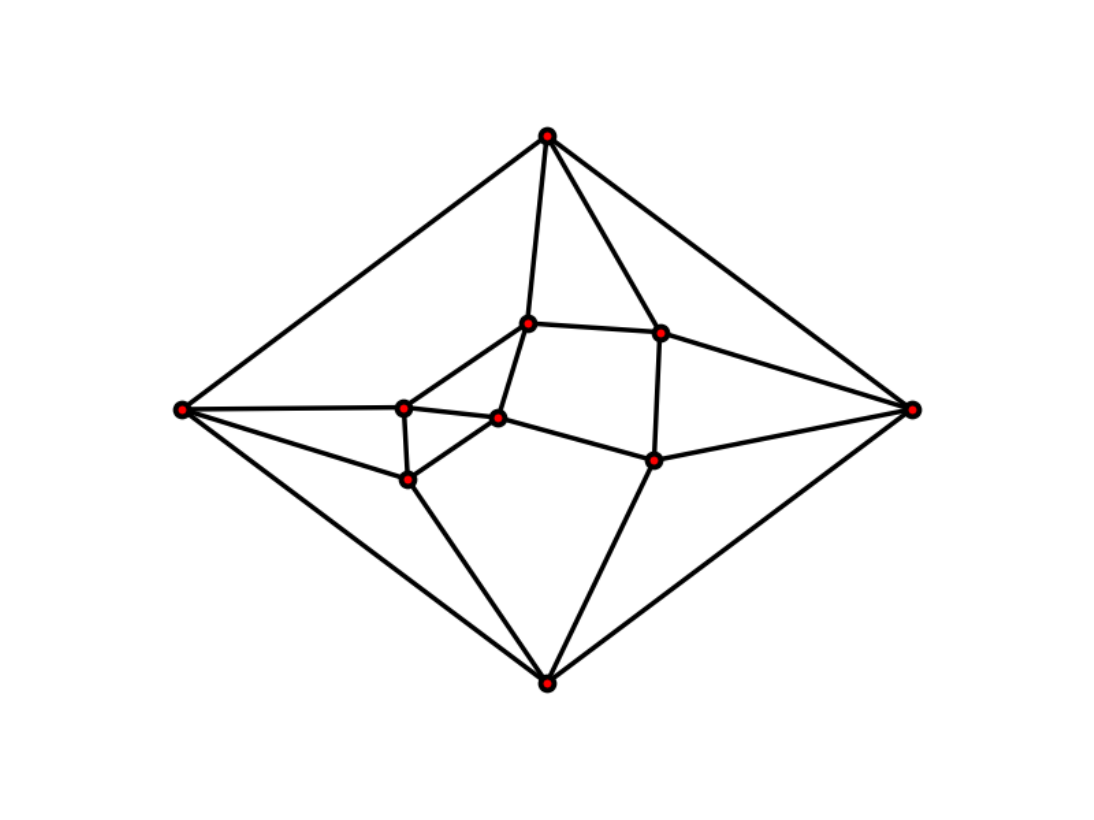} \\
(4) $\operatorname{vol} (P_{12}^{\min}) = 8.137885$  &   (5) $\operatorname{vol} (P_{12}^{\max}) =8.612415$ \\
\includegraphics[width=0.4\textwidth]{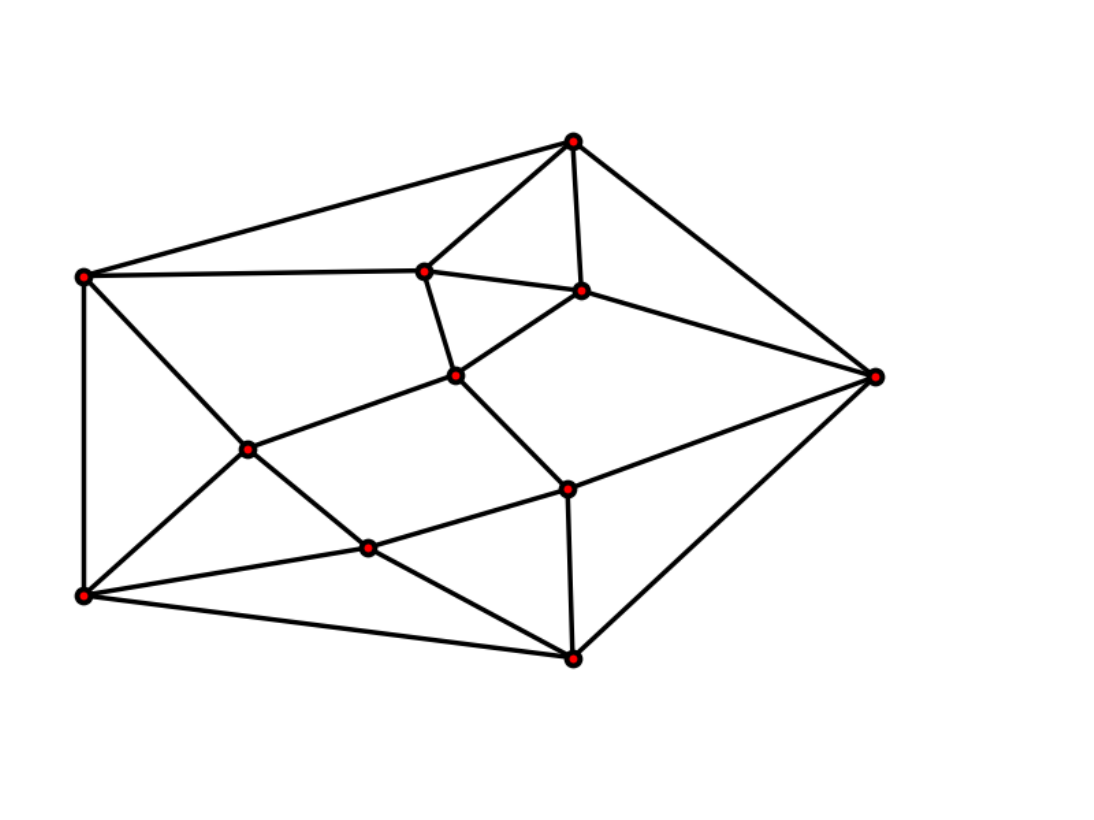} & 
\includegraphics[width=0.4\textwidth]{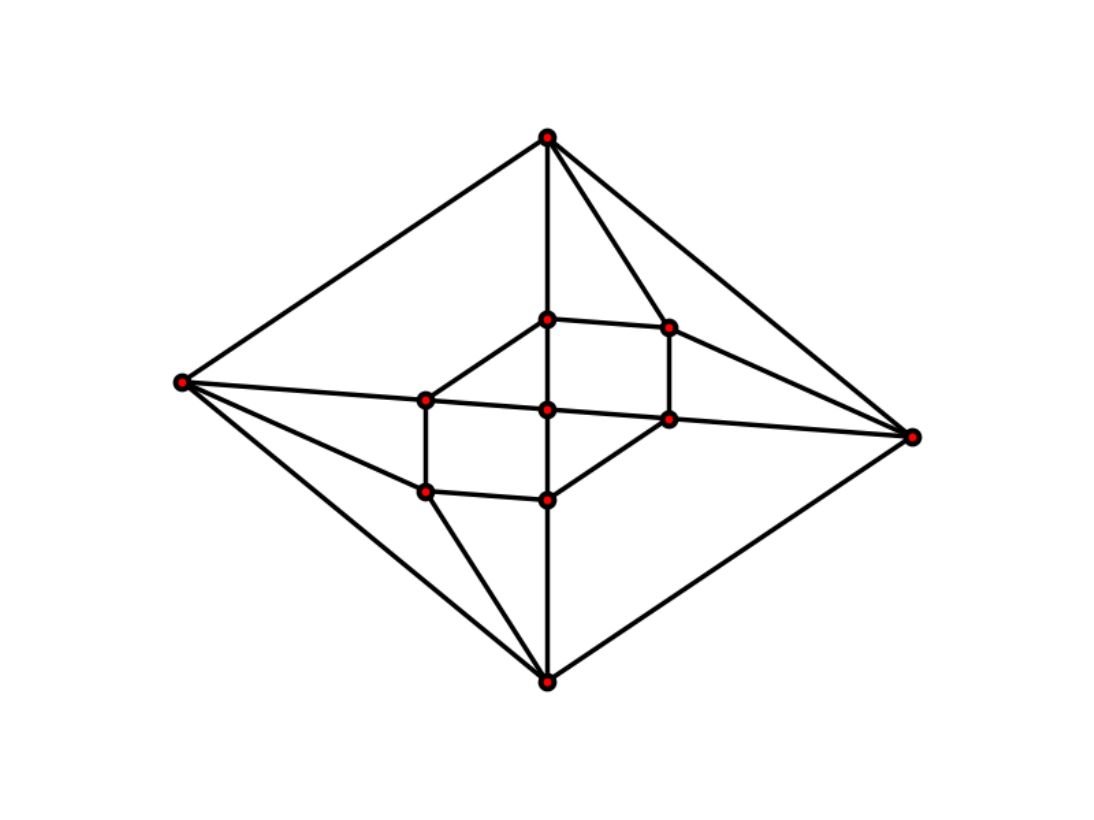} \\ 
(6) $\operatorname{vol} (P_{13}^{\min}) =  9.686908 $ & (7) $\operatorname{vol} (P_{13}^{\max}) =10.149416$  \\  
\end{tabular}
\end{table} 

The volume calculations showed that the following fact holds for $n \leq 23$. If $n$ is even, then the smallest volume is achieved on the antiprism, that is, $P_{n}^{\min} = A (k)$, where $n = 2k+2$. If $n$ is odd, then the smallest volume is achieved on the twisted antiprism, that is, $P_{n}^{\min} = A (k)^{*}$, where $n = 2k+3$. We formulate this observation in the following form. 

\begin{proposition} \label{prop}
For ideal right-angled hyperbolic polyhedra with at most 23 faces, the minimum value of volumes is achieved on antiprisms and twisted antiprisms. The minimum and maximum volumes are presented in Tables~\ref{table-2}, \ref{table-40} and~\ref{table-50}.
\end{proposition}

\begin{conjecture} \label{hyp2.1} 
If $n \geq 8$ is even, then the minimum volume is achieved on the antiprism, that is, $P_{n}^{\min} = A (k)$, where $n = 2k+2$. If $n\geq 11$ is odd, then the minimum volume is achieved on the twisted antiprism, that is, $P_{n}^{\min} = A (k)^{*}$, where $n = 2k+3$.
\end{conjecture}

The statement~\ref {prop} confirms the conjecture for $n\leq 23$.

\begin{table}[!ht] \label{table-4} 
\caption{Ideal right-angled polyhedra with $n$ faces having minimum and maximum volume, $14 \leq n \leq 18$.} \label{table-40}
\begin{tabular}{cccc} 
$P_{14}^{\min}$ & \includegraphics[width=0.4\textwidth]{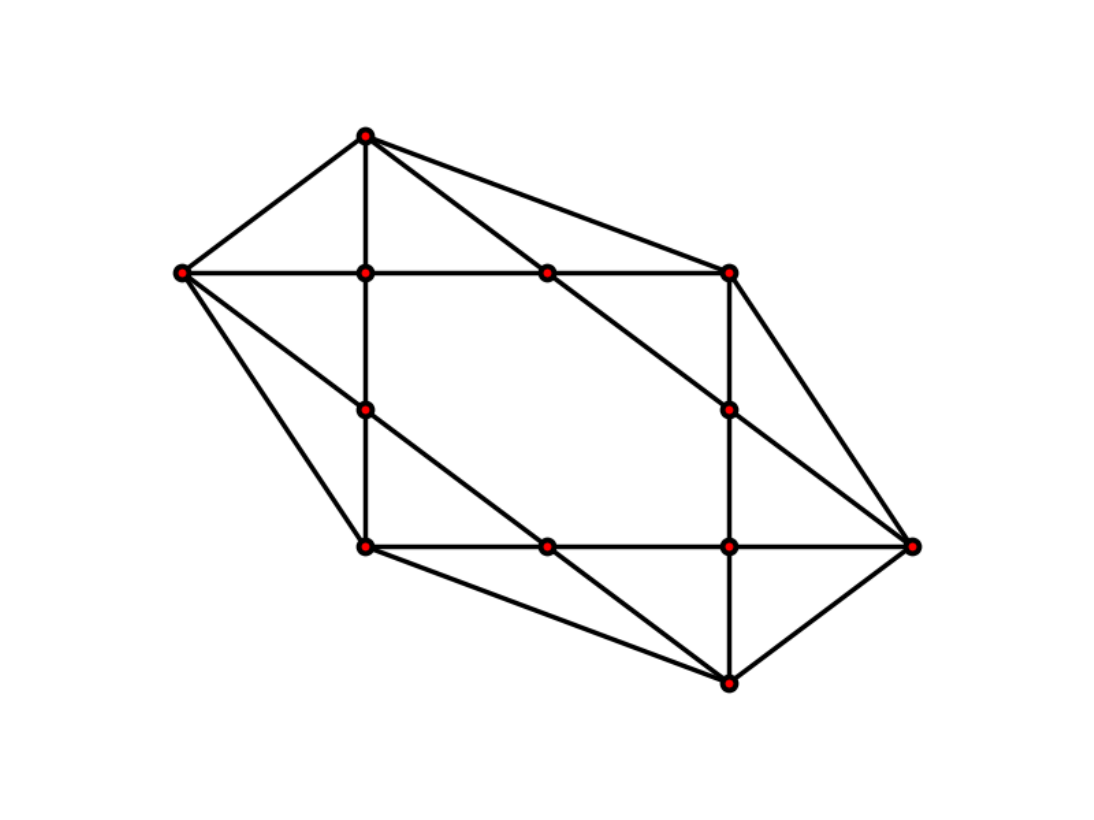} & $P_{14}^{\max}$ & \includegraphics[width=0.4\textwidth]{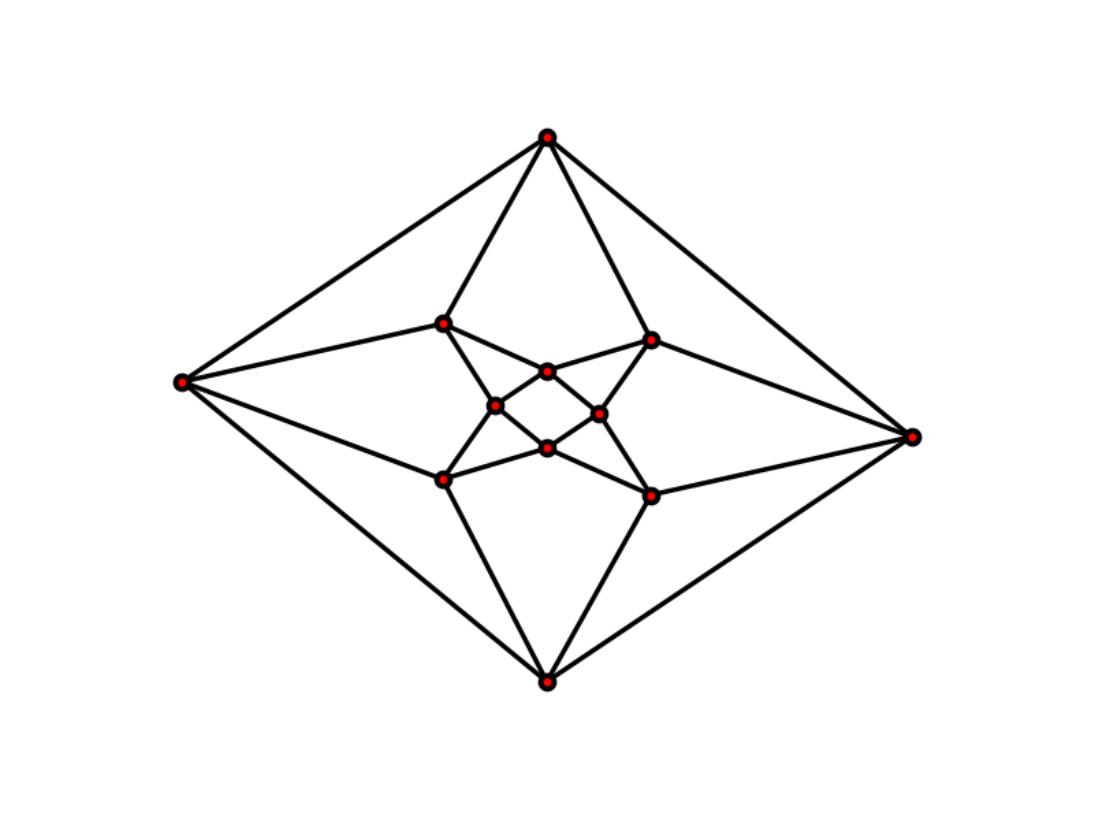} \\ 
$P_{15}^{\min}$ & \includegraphics[width=0.4\textwidth]{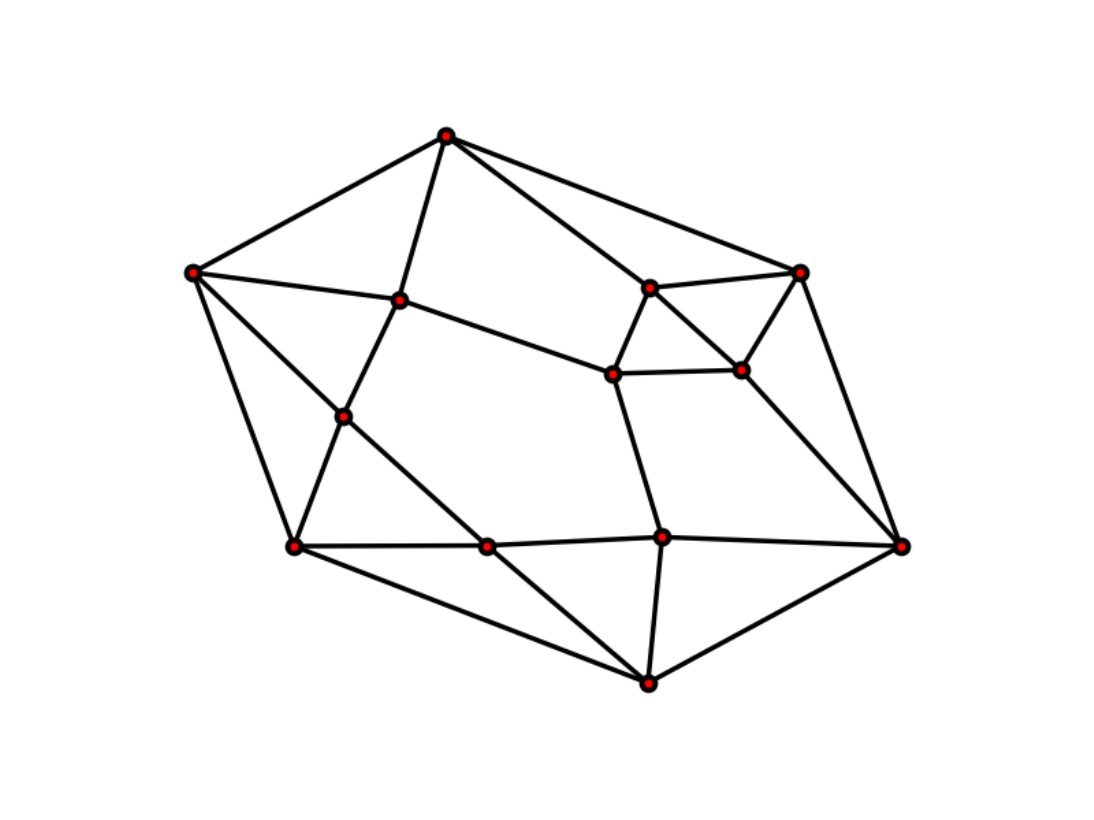} & $P_{15}^{\max}$ & \includegraphics[width=0.4\textwidth]{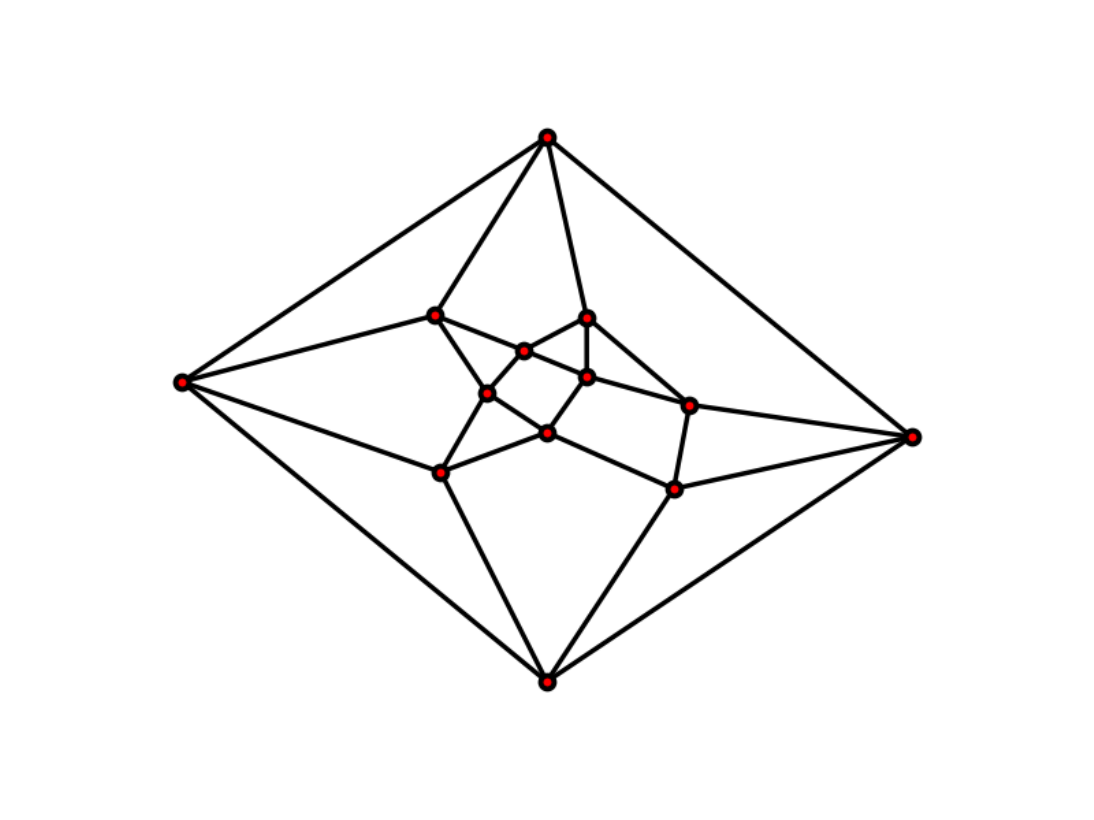} \\ 
$P_{16}^{\min}$ & \includegraphics[width=0.4\textwidth]{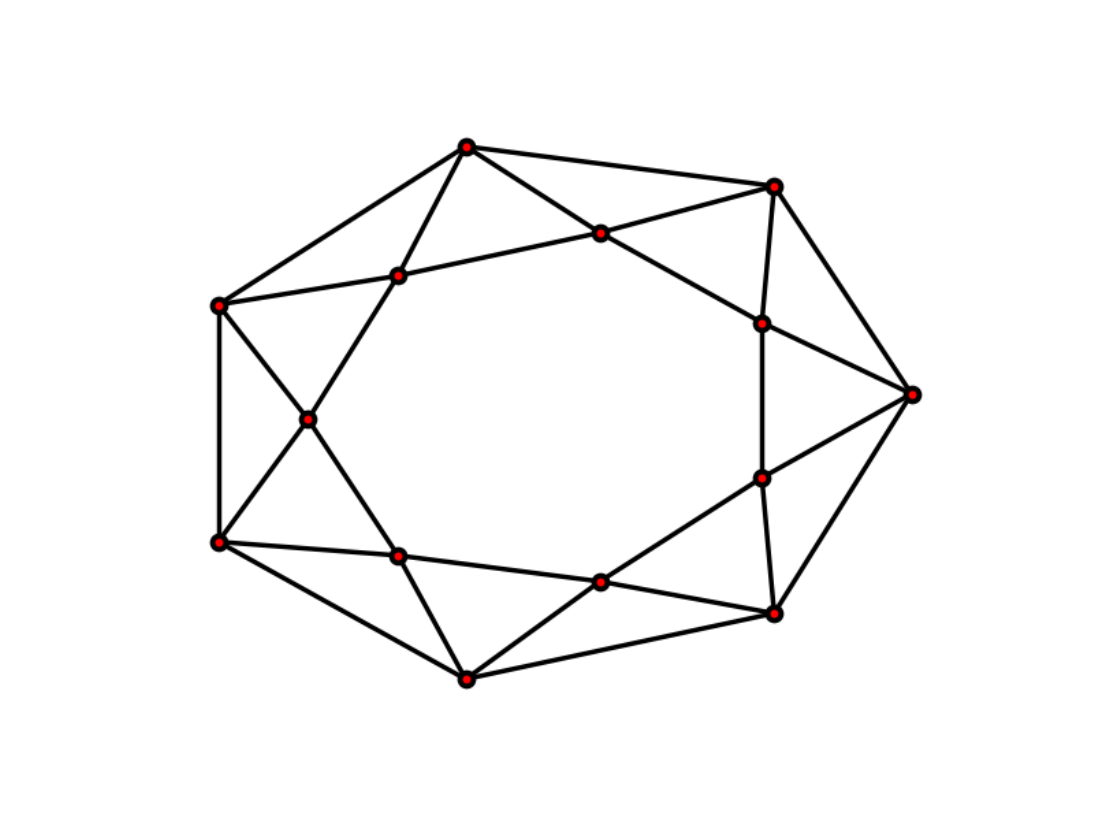} & $P_{16}^{\max}$ & \includegraphics[width=0.4\textwidth]{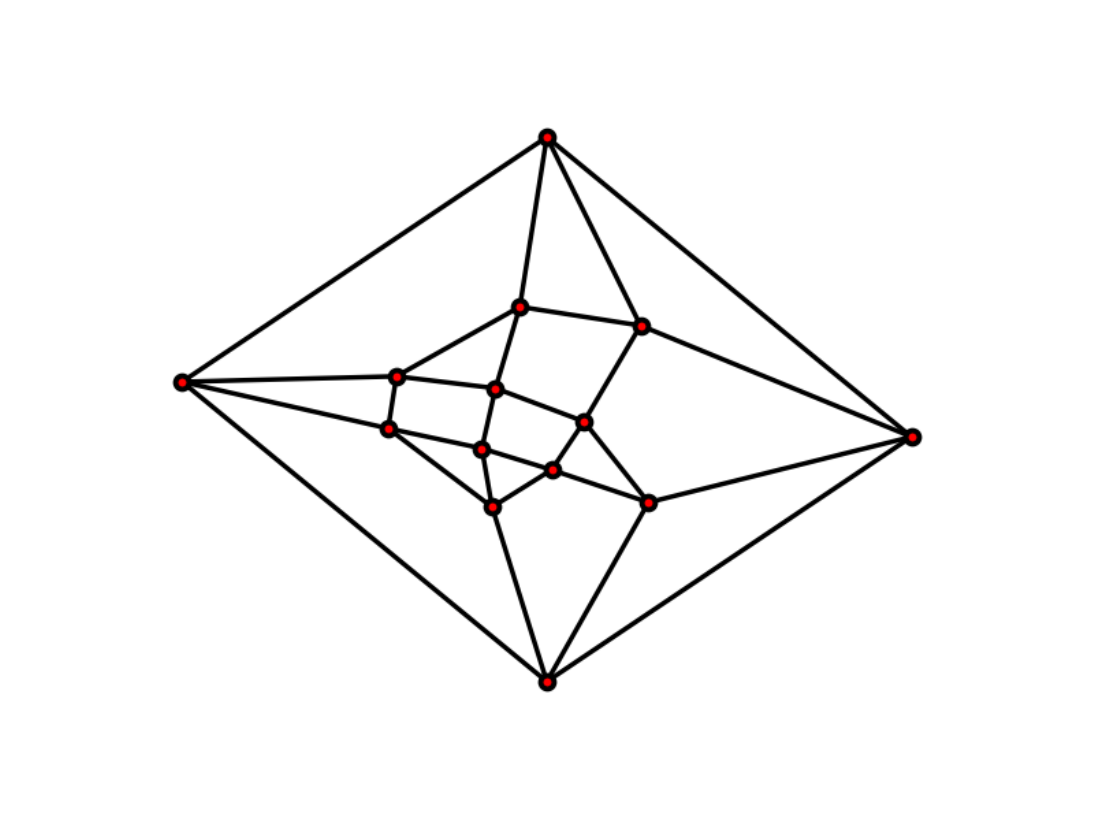} \\ 
$P_{17}^{\min}$ & \includegraphics[width=0.4\textwidth]{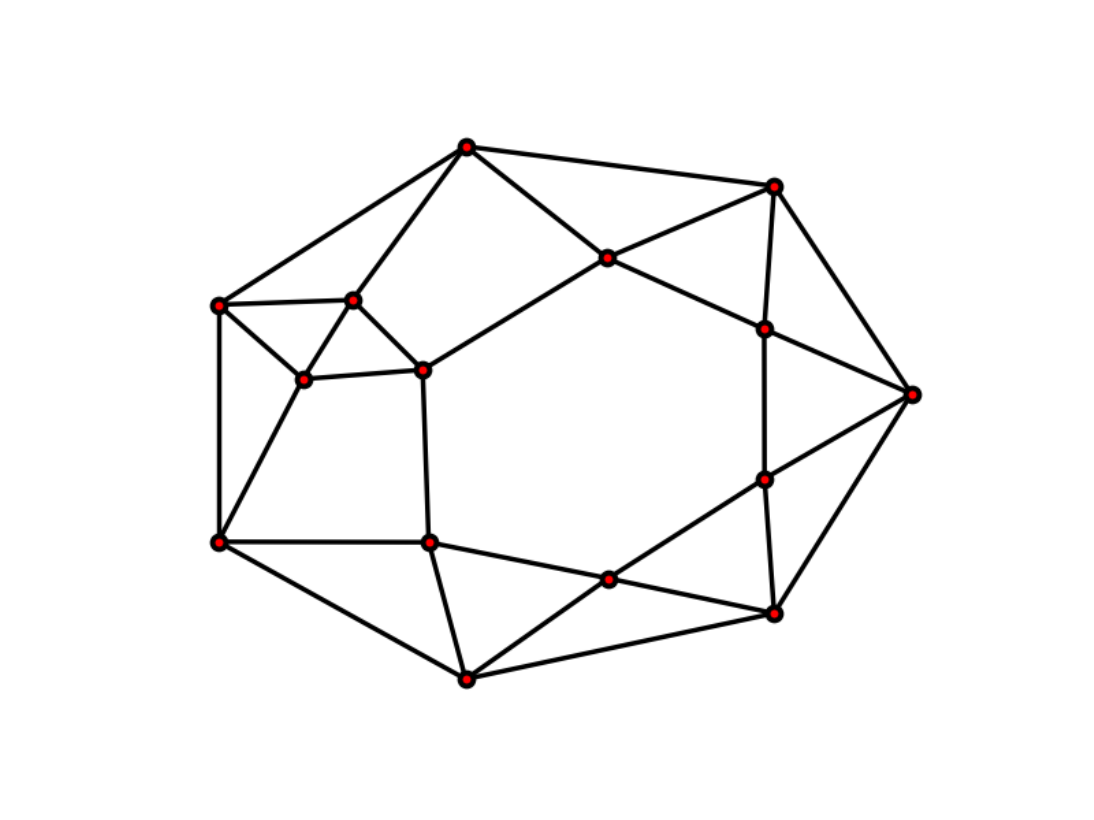} & $P_{17}^{\max}$ & \includegraphics[width=0.4\textwidth]{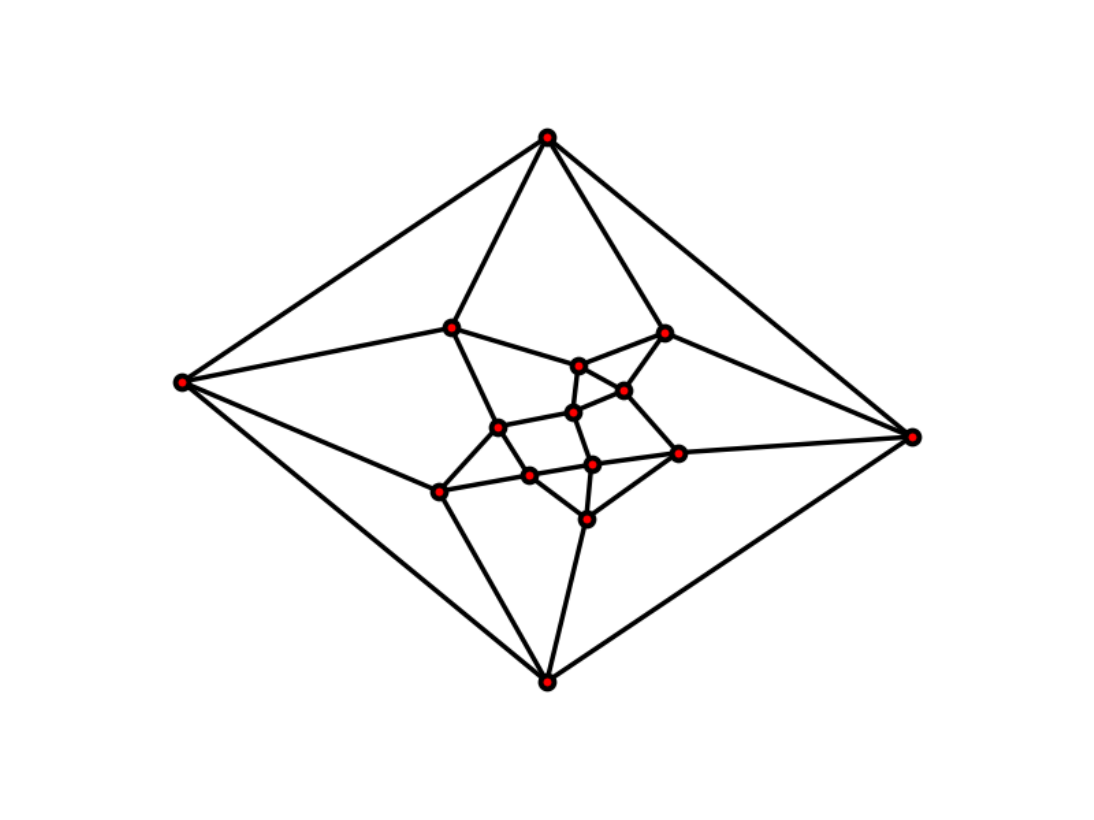} \\ 
$P_{18}^{\min}$  & \includegraphics[width=0.4\textwidth]{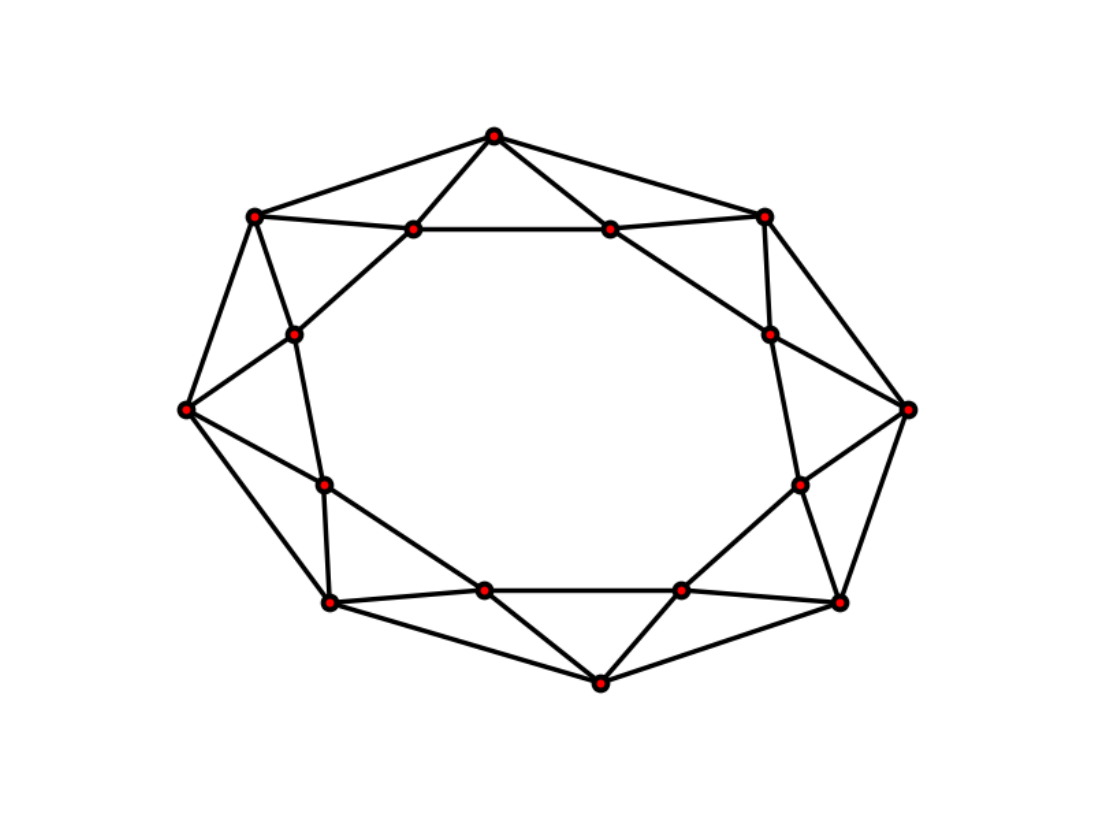} & $P_{18}^{\max}$ & \includegraphics[width=0.4\textwidth]{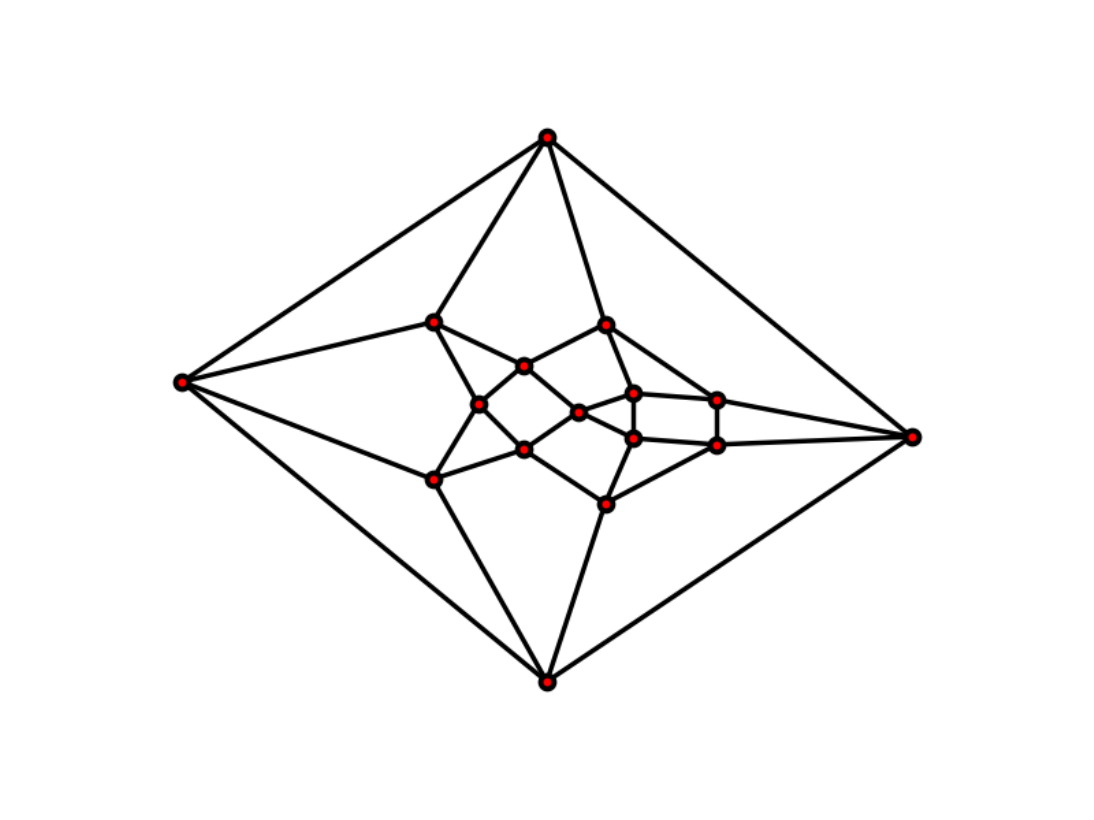} 
\end{tabular}
\end{table} 

\begin{table}[!ht] \label{table-5} 
\caption{Ideal right-angled polyhedra with $n$ faces having minimum and maximum volume, $19 \leq n \leq 23$.} \label{table-50}
\begin{tabular}{cccc} 
$P_{19}^{\min}$ & \includegraphics[width=0.4\textwidth]{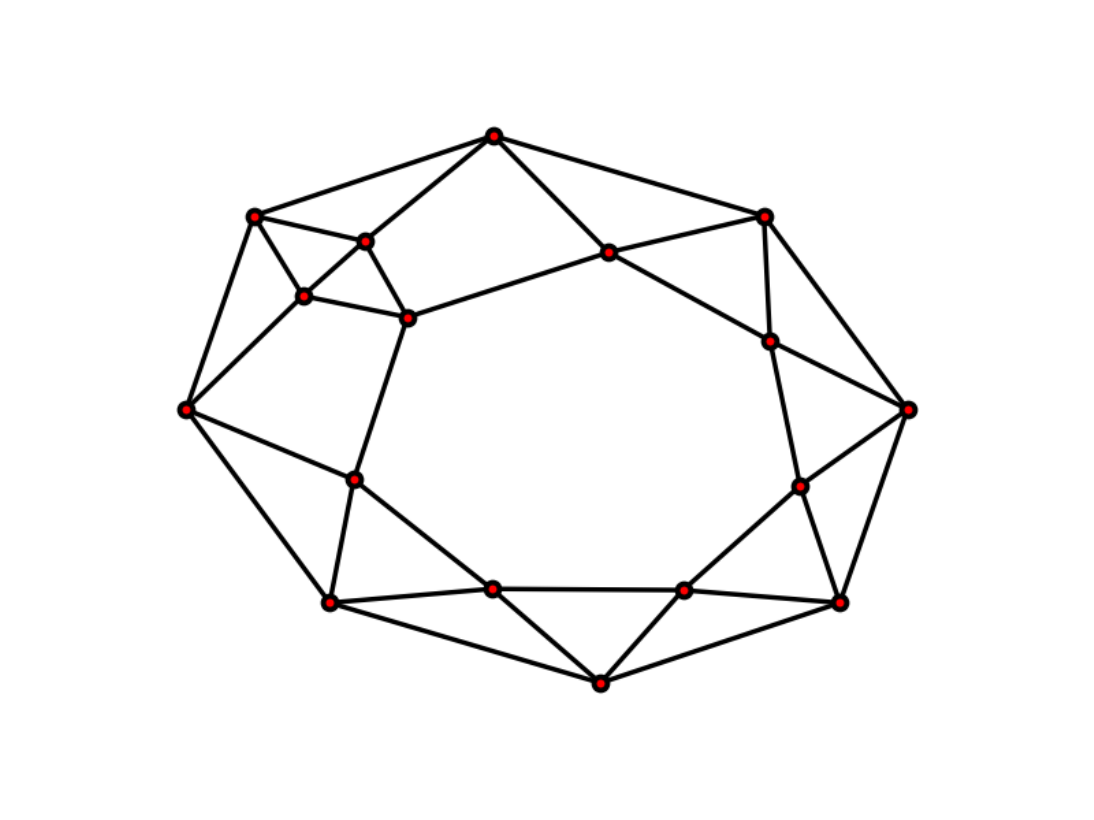} & $P_{19}^{\max}$ & \includegraphics[width=0.4\textwidth]{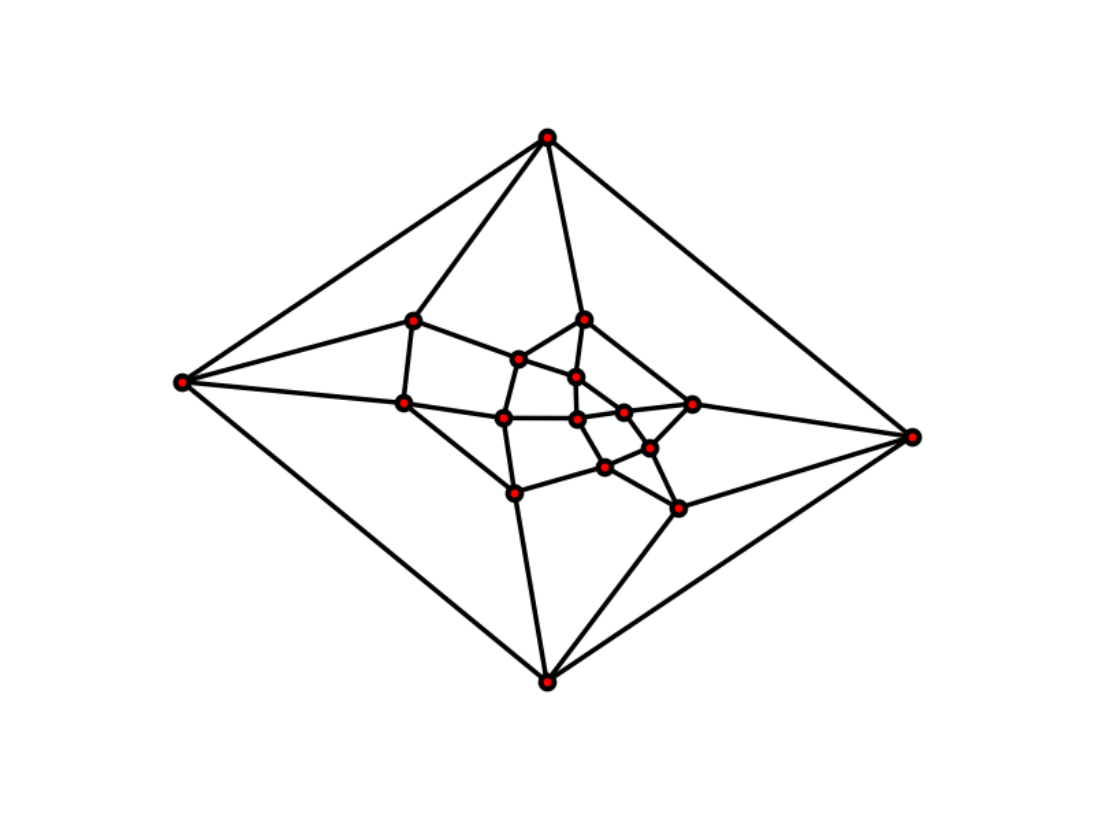} \\ 
$P_{20}^{\min}$ & \includegraphics[width=0.4\textwidth]{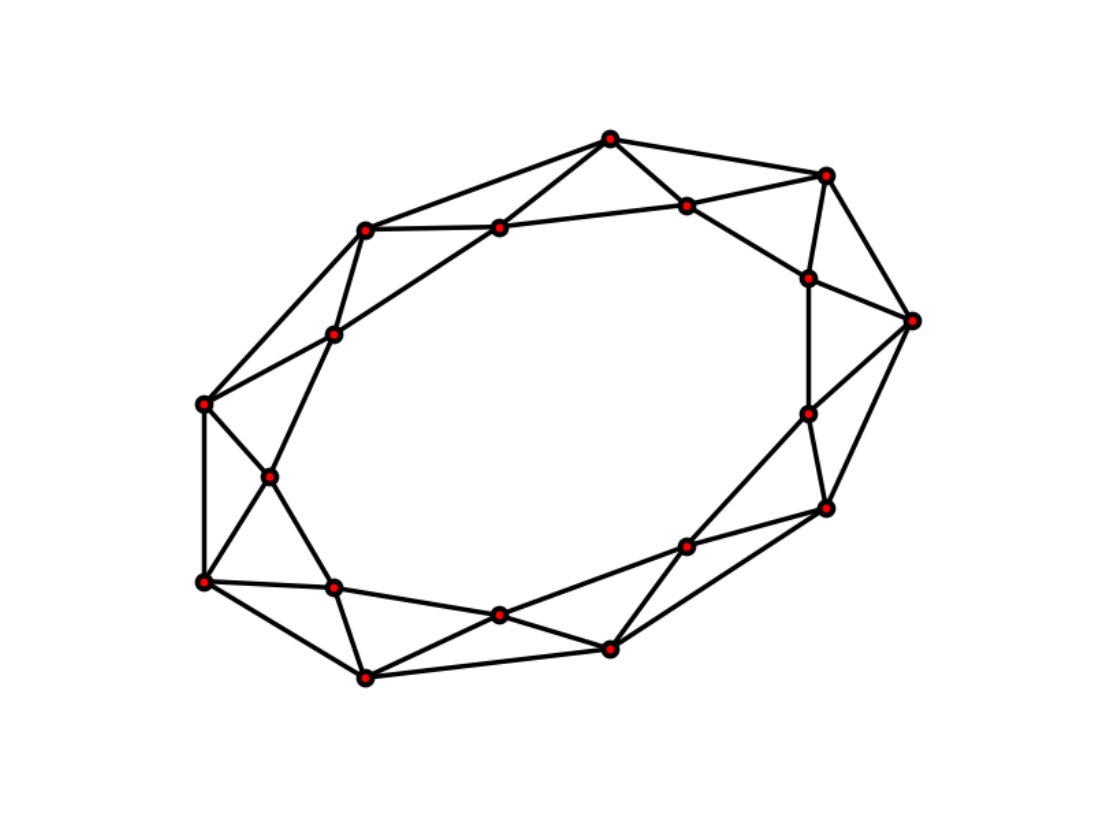} & $P_{20}^{\max}$ & \includegraphics[width=0.4\textwidth]{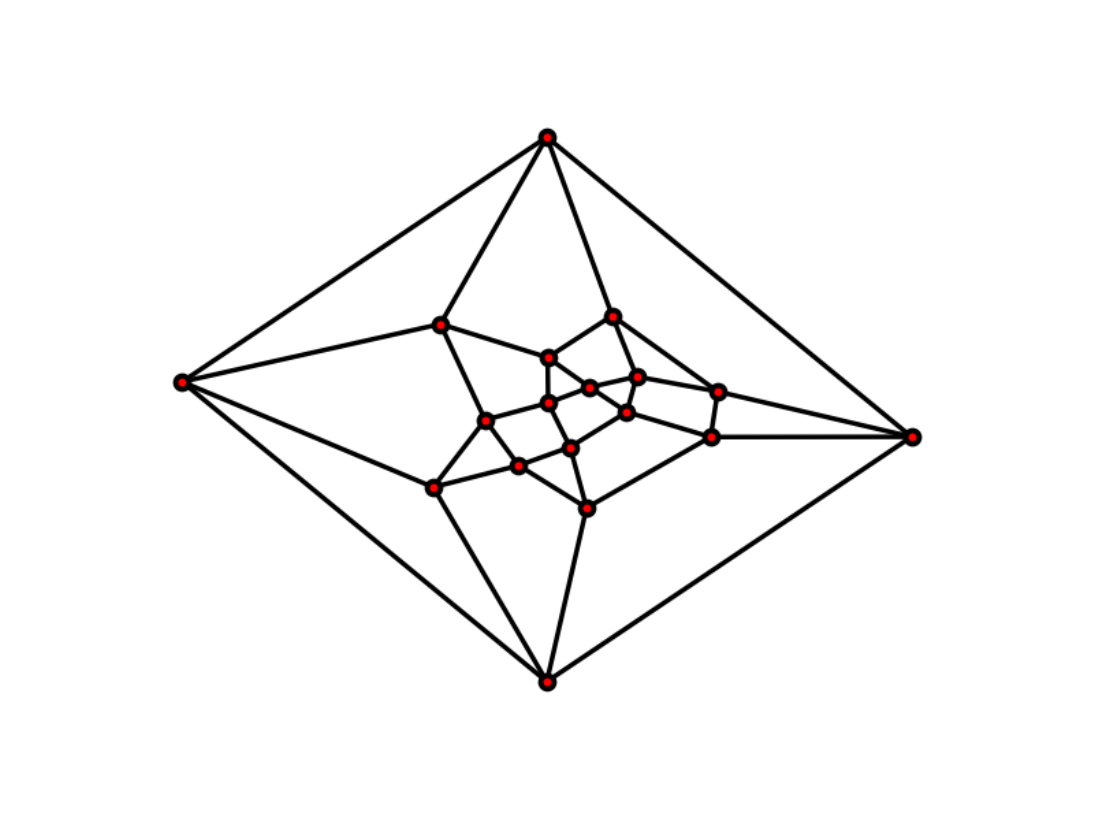} \\ 
$P_{21}^{\min}$ &  \includegraphics[width=0.4\textwidth]{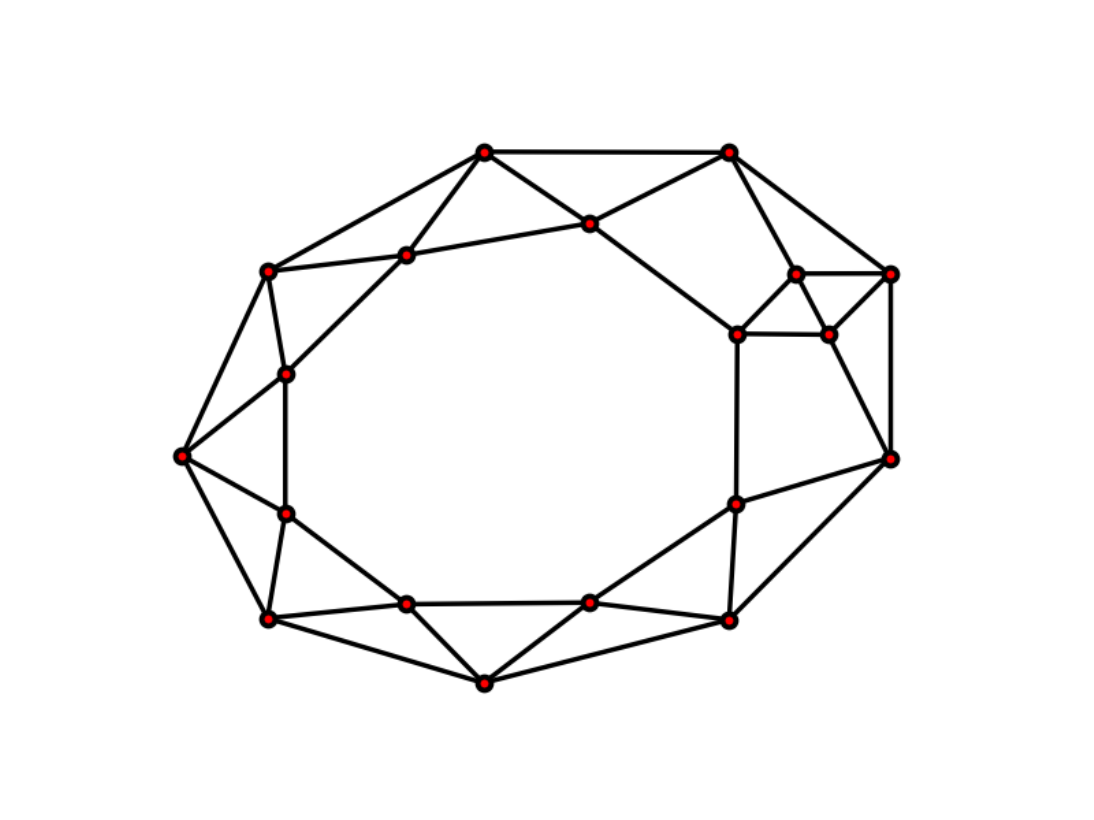} &  $P_{21}^{\max}$ & \includegraphics[width=0.4\textwidth]{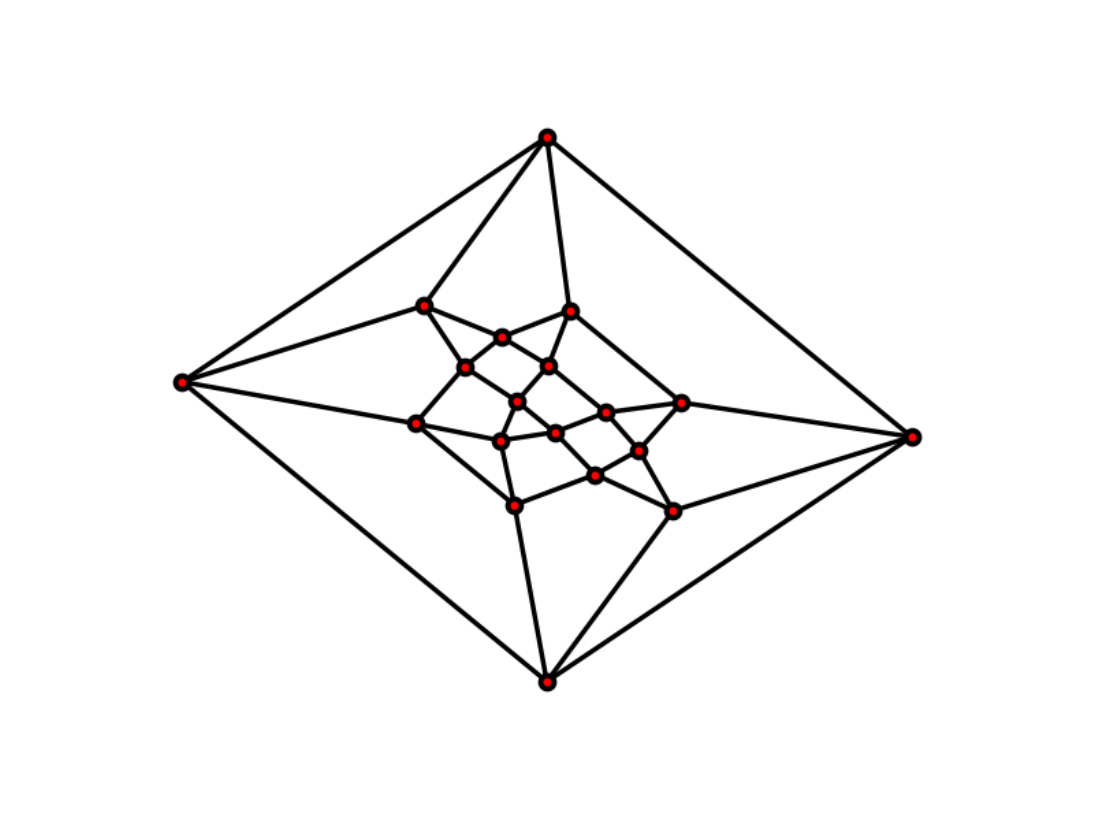} \\ 
$P_{22}^{\min}$ &  \includegraphics[width=0.4\textwidth]{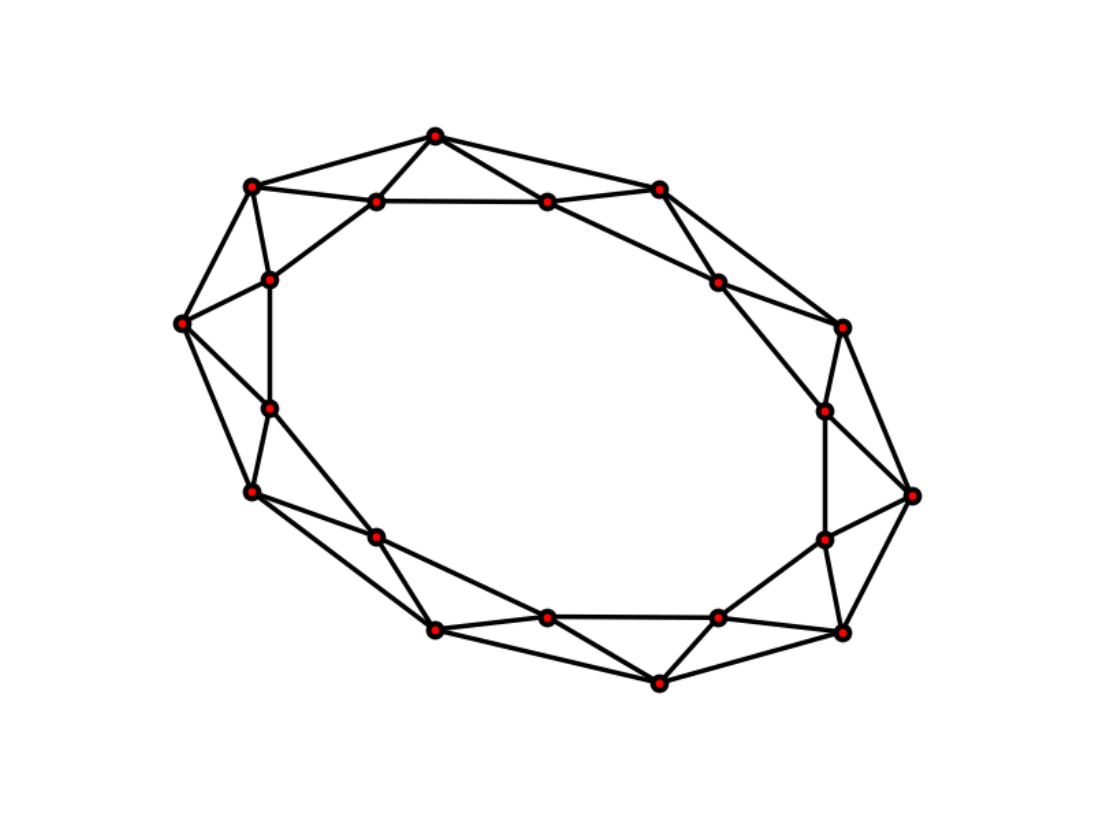} & $P_{22}^{\max}$ & \includegraphics[width=0.4\textwidth]{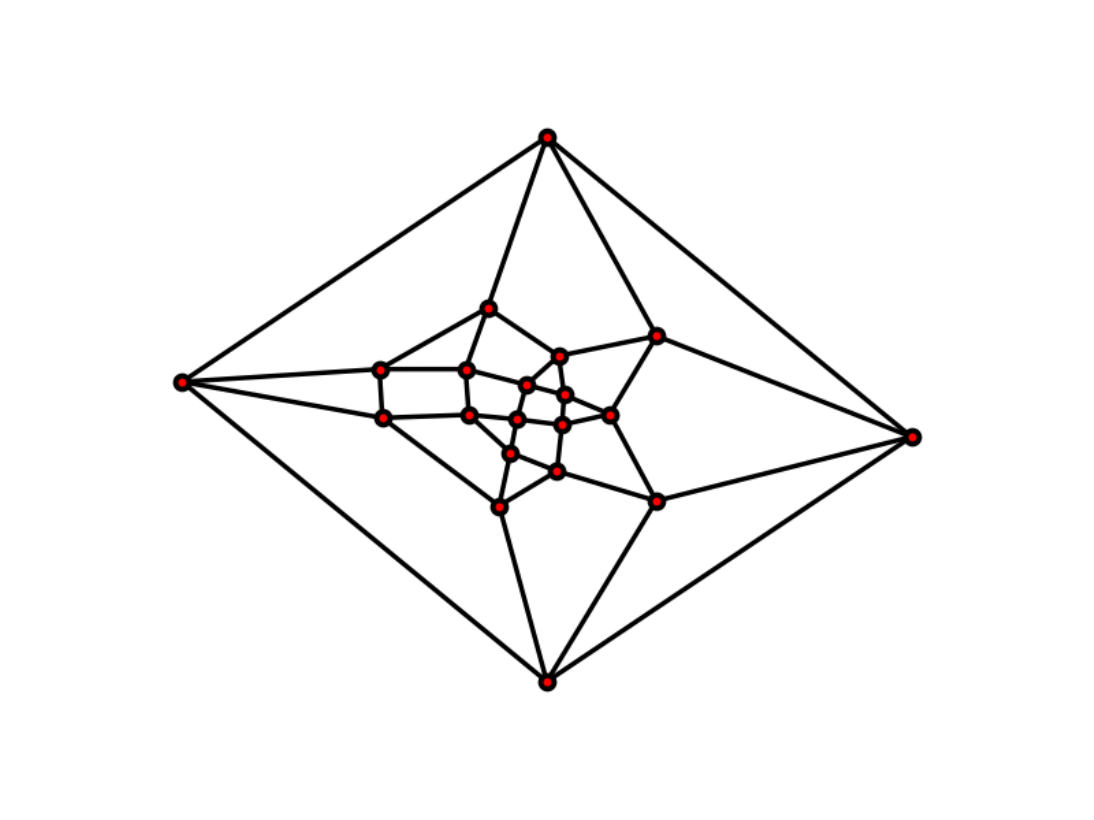} \\ 
$P_{23}^{\min}$ & \includegraphics[width=0.4\textwidth]{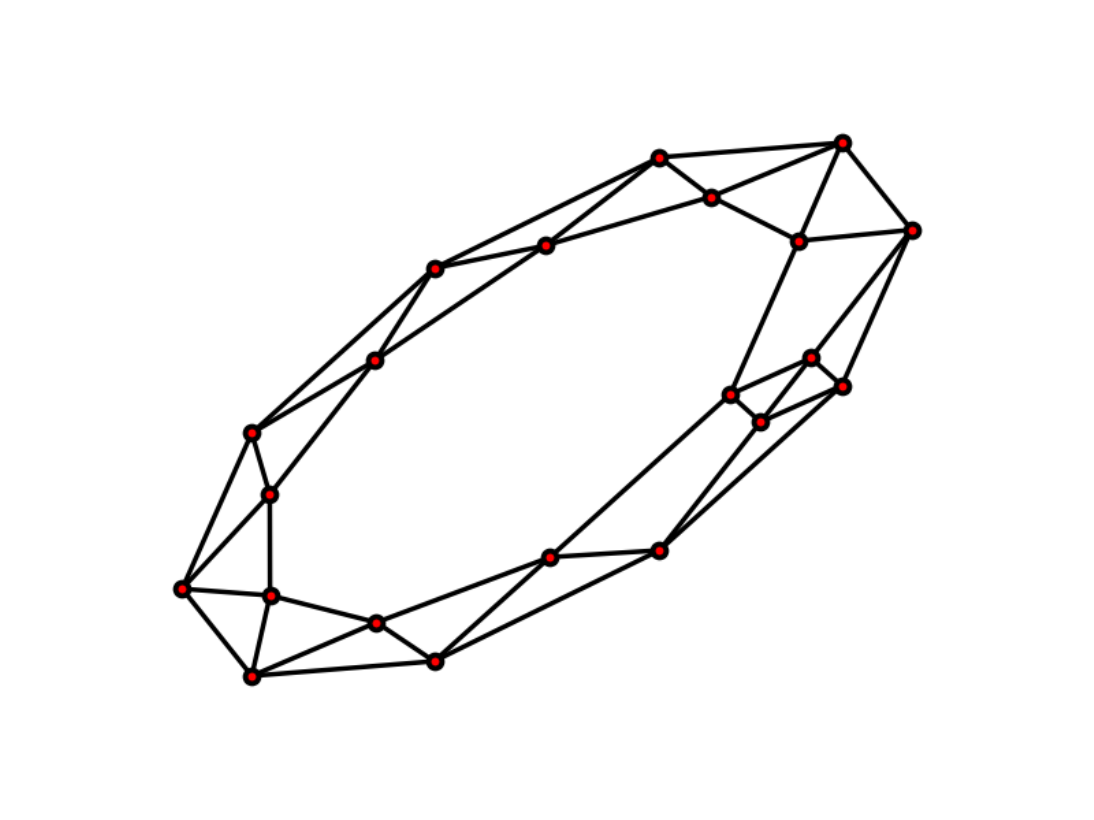} & $P_{23}^{\max}$ & \includegraphics[width=0.4\textwidth]{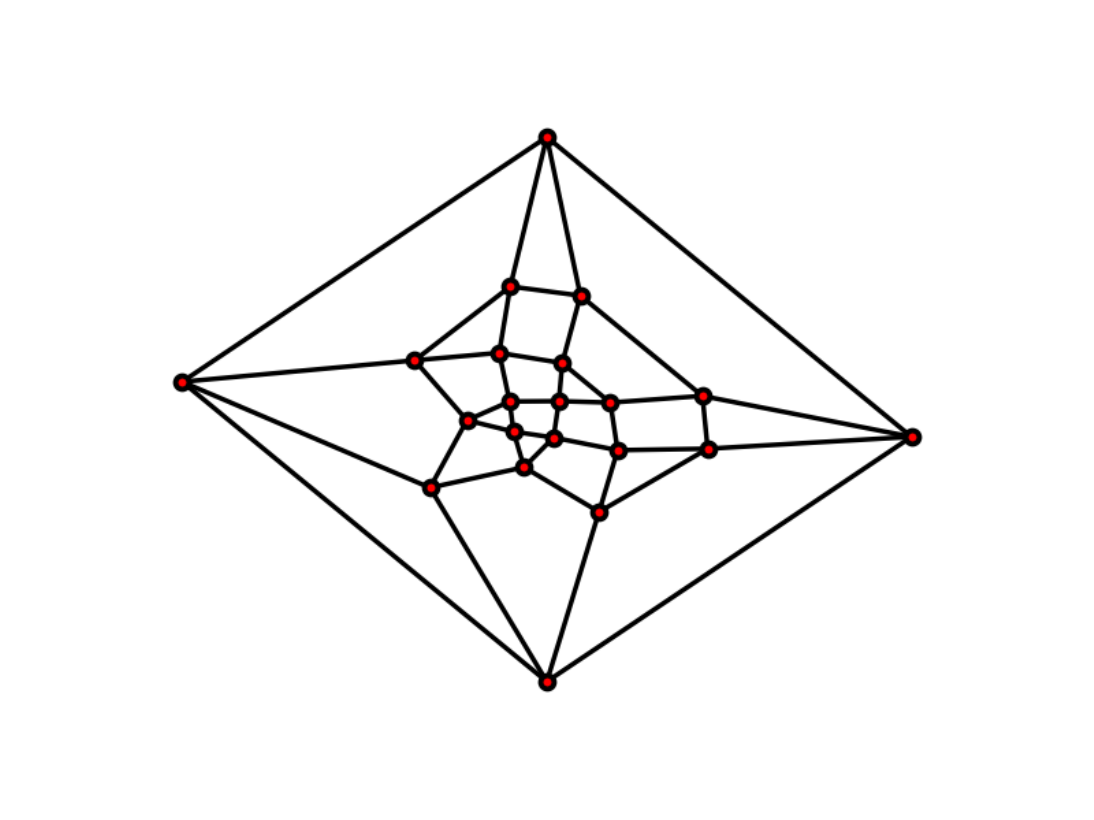} 
\end{tabular}
\end{table} 

\section{Upper and lower volume bounds} \label{sec3} 

Bilateral bounds for the volumes of ideal right-angled polyhedra in terms of the number of vertices were obtained by K.~Atkinson in~\cite{A09}.

\begin{theorem} \cite{A09}  \label{theoremAtkinson}
Let $P$ be an ideal right-angled polyhedron with $N$ vertices, then
$$
(N-2) \cdot \frac{v_8}{4} \leqslant \textrm{\rm vol} (P) \leqslant (N-4) \cdot \frac{v_8}{2}. 
$$
Both inequalities became equalities when $P$ is the regular ideal hyperbolic octahedron. Moreover, there exists a sequence of ideal right-angled polyhedra  $P_i$ with $N_i$ vertices such that  $\textrm{\rm vol} (P_i) / N_i$ tends to $v_8 / 2$ as $i \to \infty$.
\end{theorem}

Figure~\ref {fig7} shows the graphs of the upper and lower bounds from Theorem~\ref {theoremAtkinson}, the set of volume values of ideal right-angled polyhedra with at most $23$ faces, where volumes of antiprisms and twisted antiprisms are separately highlighted.

\begin{figure}[ht]
\centering
\includegraphics[width=1.0\textwidth]{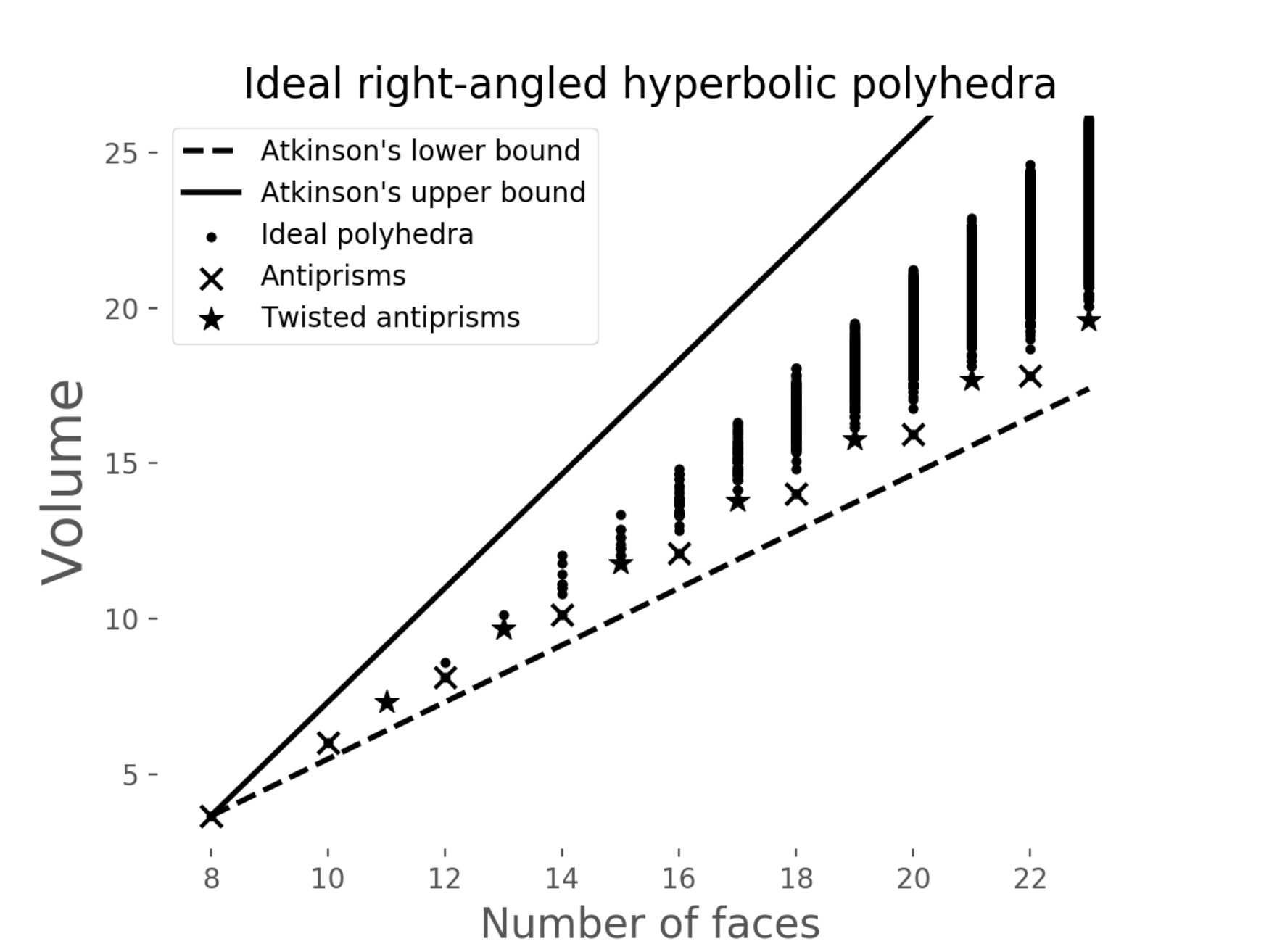}
\caption{The set of volumes and Atkinson bounds.} \label{fig7}
\end{figure}

The upper bound in Theorem~\ref {theoremAtkinson} can be improved as follows.

\begin{theorem} \label{theorem3.2} 
Let $P$ be an ideal right-angled hyperbolic polyhedron with $N$ vertices, different from the octahedron. Let $F_{1}$ and $F_ {2}$ be two faces of $P$ such that $F_ {1}$ is $n_ {1}$-gon, and $F_ {2}$ is $n_{2}$-gon, where $n_{1}, n_{2} \geq 4$. Then for its volume the following upper bound holds:
$$
\operatorname{vol} (P) \leq \left( N - \frac{n_{1}}{2} - \frac{n_{2}}{2} \right) \cdot \frac{v_{8}}{2}. 
$$
\end{theorem}

\begin{proof} As follows from the Euler formula for polyhedra, if an ideal right-angled polyhedron $P$ is different from an octahedron, then it would have two faces that have at least four sides. For definiteness, we denote these faces by $F_{1}$ and $F_ {2}$. We consider two cases according to whether the faces $F_ {1}$, $F_ {2}$ are adjacent or not.

(1) Let the faces $F_ {1}$ and $F_ {2}$ be not adjacent. We construct a family of right-angled polyhedra by induction, attaching at each step a copy of the polyhedron $P$. Put $P_ {1} = P$. Define $P_{2}=P_ {1} \cup_{F_{1}}P_{1}$, identifying two copies of the polyhedron $P_ {1}$ along the face $F_ {1}$. Obviously, $P_ {2}$ is an ideal right-angled polyhedron with the number of vertices $N_ {2} = 2N-n_ {1}$ and the volume  $\operatorname{vol} (P_{2}) = 2 \operatorname{vol} (P)$. The polyhedron $P_{2}$ has at least one face isometric to $F_{2}$. We attache the polyhedron $P$ to the polyhedron $P_{2}$ along this face. We get $P_{3} = P_{2} \cup_{F_{2}} P = P \cup_{F_{1}} P \cup_{F_{2}} P$. Obviously, $P_{3}$ is an ideal right-angled polyhedron with the number of vertices $N_{3} = 3N - n_{1} - n_{2}$ and the volume $\operatorname{vol} (P_{3}) = 3 \operatorname{vol} (P)$. Continuing the process of adding the polyhedron $P$ alternately through the faces isometric to $F_{1}$ and $F_{2}$, we obtain the polyhedron $P_{2k+1} = P_{2k-1} \cup_{F_{1}} P \cup_{F_{2}} P$, which is an ideal right-angled polyhedron with $N_{2k+1} = (2k+1) N - k, n_{1} - k n_{2}$ vertices and of volume $\operatorname{vol} (P_{2k+1}) = (2k+1) \operatorname{vol} (P)$.
Now let us  apply the upper bound from Theorem~\ref{theoremAtkinson} to polyhedron $P_{2k+1}$:
$$
(2k+1) \operatorname{vol} (P) \leq \left( (2k+1) N - k n_{1} - k n_{2} - 4 \right) \frac{v_{8}}{2}.  
$$
Dividing both sides of the inequality by $(2k+1)$ and passing to the limit as $k \to \infty $, we obtain the required inequality.

(2) Let the faces $F_{1}$ and $F_{2}$ be adjacent. Put $P_{2}=P \cup_{F_{1}} P$. The constructed polyhedron $P_{2}$ has $N_{2}=2N-n_{1}$ vertices and its volume is two times the volume of the polyhedron $P$. By construction, the polyhedron $P_{2}$ has a face $F_{21}$, which is a $(2n_{2}-2)$-gon. Since the face of $F_1$ has at least 4 edges, there is a face in $P$ adjacent to $F_{1}$, but not adjacent to $F_{2}$. As a result of attaching $P$ along $F_{1}$, this face will turn into a face $F_{22}$ in a polyhedron $P_{2}$ that has at least $4$ sides.  Thus, in $P_{2}$ there is a pair of non-adjacent faces $F_{21}$ and $F_{22} $, each of which has at least $4$ sides. This situation corresponds to the already proved case (1). Thus, for the polyhedron $P_{2}$ and its non-adjacent faces $F_{11}$ and $F_{12}$ we get:
$$
2 \operatorname{vol} (P) \leq \left( N_{2} - \frac{(2n_{2} -2)}{2} - \frac{4}{2} \right) \frac{v_{8}}{2}, 
$$
from where we receive
$$
2 \operatorname{vol} (P) \leq \left( 2N-n_{1}-n_{2}  - 1\right) \frac{v_{8}}{2}. 
$$
and therefore, 
$$
\operatorname{vol} (P) < \left( N-\frac{n_{1}}{2}-\frac{n_{2}}{2} \right) \frac{v_{8}}{2}. 
$$
\end{proof} 

\begin{theorem}  \label{theorem3.3} 
Let $P$ be an ideal right-angled hyperbolic polyhedron with $N \geq 17$ faces and having only triangular or quadrilateral faces. Then for its volume the following upper bound holds:
$$
\operatorname{vol} (P) < (N-5) \frac{v_{8}}{2}
$$ 
\end{theorem} 

\begin{proof}
Observe  that in the polyhedron $P$ there are three quadri\-la\-teral faces $F_{1}$, $F_{2}$, $F_{3}$ such that $F_{2}$ is adjacent to $F_{1}$ and $F_{3}$ both. In fact, assume that there is no such triple of faces. Then each quadrilateral face is adjacent to at most one quadrangular face. If a quadrilateral face has no adjacent quadrilateral (we will say that it is isolated), then through its four sides it is adjacent to the triangular faces. If two quadrilaterals are adjacent to each other and none of them is adjacent to another quadrilateral (we will say that the faces form a pair), then their union is adjacent through six sides with triangular faces. Hence, if there are $k_{1}$ isolated quadrilateral  faces and $k_{2}$ pairs of quadrilateral faces, then through their sides they are adjacent to triangular faces through $4 n_{1} + 6 n_{2}$ sides. Since the polyhedron does not contain $n$-gonal faces for $n \geq 5$, it follows from Euler's formula that the number of triangles is $8$. Their total number of sides is $24$. If the number of faces $N \geq 17$ and eight of them are triangles, then $n_{1} + 2 n_{2} \geq 9$ and $S = 4 n_{1} + 6 n_{2}$ sides of triangles are required.  Using the fact that $2n_{2} \geq 9 - n_{1}$, we obtain $S \geq 4 n_{1} + 3(9-n_{1}) = 27 + n_{1} > 24$. This contradiction implies that there is a triple of sequentially adjacent quadrilateral faces $F_{1}$, $F_{2}$, $F_{3}$, where $F_{2}$ is adjacent to $F_{1}$ and $F_{3}$

Let us consider the union $P_{2} = P \cup_{F_{2}} P$ of two copies of $P$ along $F_{2}$. Then the doubled faces $F_{1}$ and $F_{3}$ of the polyhedron $P$ will give two hexagonal faces in the polyhedron $P_{2}$. The total number of vertices in $P_{2}$ is $2N-4
$. We apply the upper bound from Theorem~\ref{theorem2.2} to $P_{2}$ and the indicated hexagonal faces: 
$$
2 \operatorname{vol} (P) < \left( 2 N - 4 - \frac{6}{2} - \frac{6}{2} \right) \frac{v_{8}}{2}, 
$$
from where we receive
$$
\operatorname{vol} (P) < \left( N - 5 \right) \frac{v_{8}}{2}, 
$$
this is exactly what we needed to prove.
\end{proof}  

The following statement describes the structure of the initial part of the set of volumes of ideal right-angled polyhedra.

\begin{proposition} \label{prop-2}
The volume values of ideal right-angled hyperbolic polyhedra not exceeding $5 v_{8}$ are listed in Table ~\ref{table-100}. The number of such values is $248$.
\end{proposition} 

\begin{proof}
By virtue of a lower bound from Theorem~\ref {theoremAtkinson}, if the number of faces of the ideal right-angled polyhedron $P$ is $F$ (hence the number of its vertices is $F-2 $), then for its volume the lower bound holds: 
$$
(F-4) \cdot \frac{v_8}{4} \leqslant \operatorname{vol} (P). 
$$
If $F\geq 24$ then this bound is at least $5 v_ {8} $. Thus, the volume of any polyhedron with at least $24$ faces is bounded below by $5 v_ {8}=18.319312$, where the approximate value is indicated on the right-hand side. Direct calculations of volumes of polyhedra with at most  $23$ faces show that the number of values of volumes not exceeding $5 v_{8}$ is $248$, All of them are listed in Table~\ref{table-100}. 
\end{proof}

\begin{table}[!ht] \caption{The first 248 values of volumes.} \label{table-100} 
\begin{center} 
{\tiny
\begin{tabular}{|r|r||r|r||r|r||r|r||r|r|} \hline \tt 
	1 & 3,663863 & 51 & 15,46561 & 101 & 16,735095 & 151 & 17,477141 & 201 & 17,974896 \\ \hline
	2 & 6,023046 & 52 & 15,478658 & 102 & 16,744556 & 152 & 17,509421 & 202 & 17,98967 \\ \hline
	3 & 7,327725 & 53 & 15,495403 & 103 & 16,750301 & 153 & 17,516143 & 203 & 18,009307 \\ \hline
	4 & 8,137885 & 54 & 15,546518 & 104 & 16,755495 & 154 & 17,517167 & 204 & 18,026172 \\ \hline
	5 & 8,612415 & 55 & 15,654866 & 105 & 16,769779 & 155 & 17,52091 & 205 & 18,038106 \\ \hline
	6 & 9,686908 & 56 & 15,655017 & 106 & 16,780195 & 156 & 17,528985 & 206 & 18,045655 \\ \hline
	7 & 10,149416 & 57 & 15,709955 & 107 & 16,798534 & 157 & 17,530777 & 207 & 18,047625  \\ \hline
	8 & 10,806002 & 58 & 15,720116 & 108 & 16,805953 & 158 & 17,548392 & 208 & 18,058361 \\ \hline
	9 & 10,991587 & 59 & 15,77016 & 109 & 16,829048 & 159 & 17,55096 & 209 & 18,063652 \\ \hline
	10& 11,136296 & 60 & 15,795313 & 110 & 16,83204 & 160 & 17,558575 & 210 & 18,063815 \\ \hline
	11 & 11,447207 & 61 & 15,803436 & 111 & 16,855785 & 161 & 17,571217 & 211 & 18,069138 \\ \hline
	12& 11,801747 & 62 & 15,8569 & 112 & 16,864012 & 162 & 17,577434 & 212 & 18,084139 \\ \hline
	13 & 12,106298 & 63 & 15,85949 & 113 & 16,896062 & 163 & 17,5839 & 213 & 18,08961 \\ \hline
	14 & 12,276278 & 64 & 15,933385 & 114 & 16,961302 & 164 & 17,600432 & 214 & 18,092676 \\ \hline
	15 & 12,414155 & 65 & 15,94014 & 115 & 16,974442 & 165 & 17,615398 & 215 & 18,099757 \\ \hline
	16 & 12,46092 & 66 & 15,958101 & 116 & 16,98803 & 166 & 17,616542 & 216 & 18,109351 \\ \hline
	17 & 12,611908 & 67 & 15,959551 & 117 & 17,004375 & 167 & 17,633184 & 217 & 18,109786  \\ \hline
	18 & 12,854902 & 68 & 15,996629 & 118 & 17,014633 & 168 & 17,671046 & 218 & 18,128273  \\ \hline
	19 & 12,883862 & 69 & 16,049989 & 119 & 17,024507 & 169 & 17,694323 & 219 & 18,129371 \\ \hline
	20 & 13,020639 & 70 & 16,061517 & 120 & 17,061166 & 170 & 17,701559 & 220 & 18,133727 \\ \hline
	21 & 13,310579 & 71 & 16,078017 & 121 & 17,061237 & 171 & 17,70449 & 221 & 18,144299 \\ \hline
	22 & 13,350771 & 72 & 16,158579 & 122 & 17,061342 & 172 & 17,709902 & 222 & 18,152718 \\ \hline
	23 & 13,447108 & 73 & 16,172462 & 123 & 17,110971 & 173 & 17,712742 & 223 & 18,15859 \\ \hline
	24 & 13,677298 & 74 & 16,213678 & 124 & 17,140322 & 174 & 17,740113 & 224 & 18,167534 \\ \hline
	25 & 13,714015 & 75 & 16,27577 & 125 & 17,159342 & 175 & 17,751064 & 225 & 18,1677 \\ \hline
	26 & 13,813278 & 76 & 16,295989 & 126 & 17,165397 & 176 & 17,759743 & 226 & 18,173199 \\ \hline
	27 & 13,907355 & 77 & 16,324638 & 127 &17,169868 & 177 & 17,766925 & 227 & 18,175729 \\ \hline
	28 & 14,030461 & 78 & 16,330917 & 128 & 17,174806 & 178 & 17,766983 & 228 & 18,180264 \\ \hline
	29 & 14,103121 & 79 & 16,331571 & 129 & 17,19799 & 179 & 17,769525 & 229  & 18,180633 \\ \hline
	30 & 14,160931 & 80 & 16,339295 & 130 & 17,199831 & 180 & 17,773653 & 230 & 18,207313 \\ \hline
	31 & 14,171606 & 81 & 16,382246 & 131 & 17,201332 & 181 & 17,790452 & 231 & 18,21988 \\ \hline
	32 & 14,273414 & 82 & 16,39832 & 132 & 17,22483 & 182 & 17,812693 & 232 & 18,233526 \\ \hline
	33 & 14,469865 & 83 & 16,448631 &133 & 17,233217 & 183 & 17,821704 & 233 & 18,234257 \\ \hline
	34 & 14,494727 & 84 & 16,465777 & 134 & 17,238195 & 184 & 17,824793 & 234 & 18,244844 \\ \hline
	35 & 14,635461 & 85 & 16,48952 & 135 & 17,280423 & 185 & 17,835469 & 235 & 18,247553 \\ \hline
	36 & 14,655449 & 86 & 16,49154 & 136 & 17,303311 & 186 & 17,83745 & 236 & 18,276848 \\ \hline
	37 & 14,766948 & 87 & 16,506891 & 137 & 17,324068 & 187 & 17,844054 & 237 & 18,281813 \\ \hline
	38 & 14,800159 & 88 & 16,518764 & 138 & 17,341161 & 188 & 17,845073 & 238 & 18,287301 \\ \hline
	39 & 14,832681 & 89 & 16,535273 & 139 & 17,342423 & 189 & 17,857212 & 239 & 18,28917 \\ \hline
	40 & 14,898794 & 90 & 16,538867 & 140 & 17,354288 & 190 & 17,860804 & 240 & 18,291323  \\ \hline
	41 & 15,031667 & 91 & 16,547725 & 141 & 17,354866 & 191 & 17,864013 & 241& 18,292895  \\ \hline
	42 & 15,052463 & 92 & 16,575188 & 142 & 17,362724 & 192 & 17,864685 & 242 & 18,299323 \\ \hline
	43 & 15,07859 & 93 & 16,595363 & 143 & 17,377493 & 193 & 17,894018 &  243 & 18,300817 \\ \hline
	44 & 15,11107 & 94 & 16,605736 & 144 & 17,377877 & 194 & 17,899631 & 244 & 18,304268 \\ \hline
	45 & 15,126498 & 95 & 16,615815 & 145 & 17,38534 & 195 & 17,901906 & 245 & 18,307302 \\ \hline
	46 & 15,169623 & 96 & 16,627568 & 146 & 17,420943 & 196 & 17,907162 & 246 & 18,31334 \\ \hline
	47 & 15,253393 & 97 & 16,657287 & 147 & 17,429408 & 197 & 17,918936 & 247 & 18,316267 \\ \hline
	48 & 15,323216 & 98 & 16,678106 & 148 & 17,45025 &198 & 17,922791 &  248 & 18,319312 \\ \hline
	49 & 15,350907 & 99 & 16,684502 & 149 & 17,470253 & 199 & 17,937276 & &  \\ \hline
	50 & 15,367058 & 100 & 16,726449 & 150 & 17,470735 & 200 & 17,944583 & &  \\ \hline
\end{tabular} 
}
\end{center}
\end{table} 

\section{Polyhedra with isolated triangles} \label{sec4} 

Recall that an ideal right-angled polyhedron has at least eight triangles. In the case of an octahedron, each triangle is adjacent to three other triangles along sides. From Table~\ref{table-40} and~\ref{table-50} one can see that with the increasing the number of faces of polyhedra with maximum volume, the triangles move away from each other more and more. From the presence of common sides, the situation changes towards the presence of common vertices. The question appears when polyhedra arise in which all triangular faces are isolated, that is, no two triangular faces have common vertices. In this case we will call the polyhedron \emph{ITR-polyhedron}, emphasizing that it satisfies the isolated triangles rule.

\begin{proposition} \label{prop4.1} 
Let $P$ be an ideal right-angled polyhedron in Lobachevsky space having $N$ faces. Denote by $p_{3}$ the number of its triangular faces. If $N < 3p_{3} +2$, then $P$ is not ITR-polyhedron.
\end{proposition} 

\begin{proof} 
Let $n$ be the maximum number of edges in faces of the polyhedron $P$. 
Denote by $p_{k}$, $k=3, \ldots , n$, the number of $k$-gonal faces in $P$. Then $\sum_{k=3}^{n} p_{k} = N$. Recall that by the formula~(\ref{formulamain}), $p_{3} = 8 + \sum_{k=4}^{n} p_{k} (k-4)$. Assume, on the contrary, that $P$ is an ITR-polyhedron. So at each vertex of the triangle there are meet three more vertices  related to the faces that are not triangular. The number of vertices of all triangles is $3 p_{3}$. So, the number of vertices in all remaining polygons must be at least $9 p_{3}$. Let us calculate this number:
$$
\begin{gathered} 
\sum_{k=4}^{n} k p_{k} = 4 p_{4} + \sum_{k=5}^{n} k p_{k} = 4 p_{4} + \sum_{k=5}^{n} p_{k} (k-4) + 4 \sum_{k=5}^{n} p_{k} \cr 
= p_{3} - 8 + 4 \sum_{k=4}^{n} p_{k} = p_{3} - 8 + 4 (N-p_{3}) = 4 N - 3 p_{3} - 8. 
\end{gathered} 
$$
Demanding inequality
$$
4 N - 3 p_{3} - 8 \geq 9 p_{3}, 
$$
we get
$$
N \geq 3p_{3} + 2,
$$
which contradicts the original condition. Therefore, for $N < 3 p_{3} +2$ the polyhedron $P$ cannot have isolated triangles.
\end{proof} 

Since the smallest possible value of $p_ {3}$ is $8$, there are no polyhedra with isolated triangles among polyhedra with at most  25 faces. But among the 26-faced polyhedra there are two such examples, which are shown in Figure~\ref{fig900}.
\begin{figure}[!ht]
\centering
\includegraphics[width=0.49\textwidth]{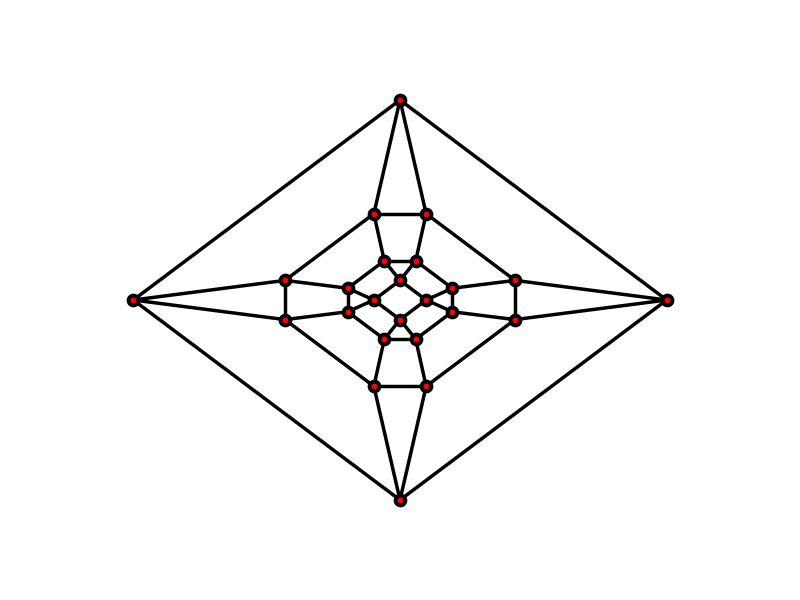}
\includegraphics[width=0.49\textwidth]{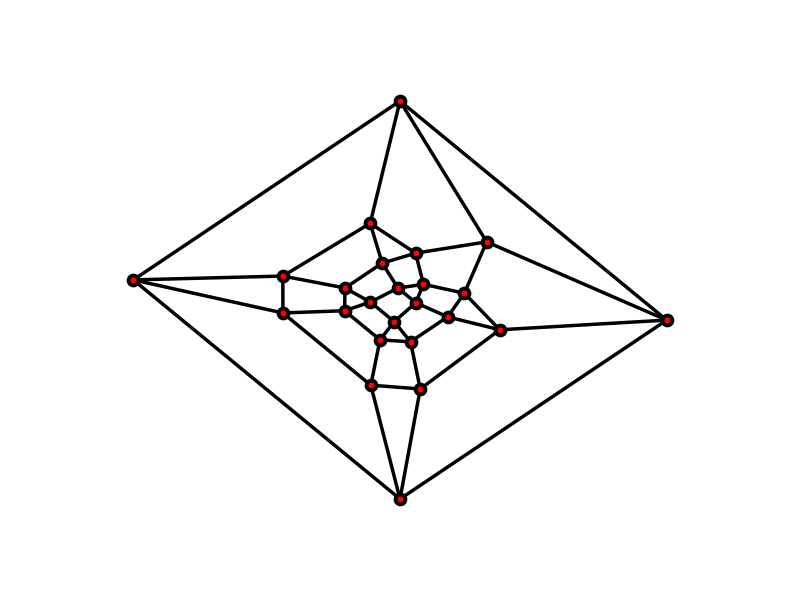} 
\caption{26-gonal polyhedra with isolated triangles} \label{fig900}
\end{figure}
The volume of the ITR-polyhedron shown  on the left-hand side is $31.0930375$, and the volume of the ITR-polyhedron shown on the right-hand side is $31.1668675$.

\begin{proposition}  \label{prop4.2} 
There are infinitely many ITR-polyhedra.
\end{proposition} 

\begin{proof} 
Let $P_{1}$ be an ITR-polyhedron (for example, it can be taken any of the polyhedra shown in Figure~\ref{fig900}) and let $F_{1}$ be one of its triangular faces.  Consider the polyhedron $P_{2} = P_{1} \cup_{F_{1}} P_{1}$ obtained by gluing two copies of the polyhedron $P_{1}$ along the face $F_{1}$. Obviously, the polyhedron $P_{2}$ will also be an ideal right-angled polyhedron with isolated triangles. Suppose, for certainty, that the faces of the polyhedron $P_{1}$, which have common sides with $F_{1}$, have $n_{1}$, $n_{2}$, $n_{3}$ sides, where $n_{i} \geq 4$, $i=1,2,3$, by virtue of the triangle isolation property. When these faces are doubled, they will turn into faces of a polyhedron $P_{2}$, having, respectively, $2n_{i}-2$ sides, $i=1, 2, 3$. If $P_{1}$ have $p_{3} (P_{1})$ triangular faces, then $P_{2}$ will have $p_{3} (P_{2}) = 2 p_{3} (P_{1}) - 2 = p_{3} (P_{1}) + \sum_{i=1}^{3} (2n_{i} - 6)$ triangular faces.
\end{proof}

\end{document}